\documentclass[11pt]{article}
\usepackage[toc,page]{appendix}

\usepackage{bbm}
\usepackage{amsmath,amsthm,latexsym,amssymb}
\usepackage{enumerate}
\usepackage{algorithm,enumerate}
\usepackage{algpseudocode}
\usepackage{hyperref}
\allowdisplaybreaks

\usepackage{color}\usepackage{graphicx}

\def\cA{{\mathcal A}}
\def\cB{{\mathcal B}}

\def\cU{{\mathcal U}}

\newcommand{\set}[1]{\left\{#1\right\}}

\def\cP{\mathcal{P}}
\def\cQ{\mathcal{Q}}
\def\cR{\mathcal{R}}

\def\ii_(#1,#2){i_{#1}^{#2}}

\def\cE{\mathcal{E}}
\def\cF{\mathcal{F}}

\def\cR{\mathcal{R}}

\renewcommand{\Pr}{\operatorname{\bf Pr}}
\newcommand\bfrac[2]{\left(\frac{#1}{#2}\right)}

\newtheorem{theorem}{Theorem}[section]

\newtheorem{lemma}[theorem]{Lemma}
\newtheorem{remark}[theorem]{Ramark}
\newtheorem{corollary}[theorem]{Corollary}

\newtheorem{observation}[theorem]{Observation}

\newtheorem{notation}[theorem]{Notation}
\date{}

\usepackage{tikz}
\usetikzlibrary{arrows}
\usetikzlibrary{decorations}
\usetikzlibrary{shapes.misc}
\usetikzlibrary{decorations.pathreplacing}

\title{A fast algorithm on average for solving the Hamilton Cycle problem}
\author{Michael Anastos\footnote{This project has received funding from the European Union’s Horizon 2020 research and innovation programme under the Marie Sk\l{}odowska-Curie grant agreement No 101034413\includegraphics[width=4.5mm, height=3mm]{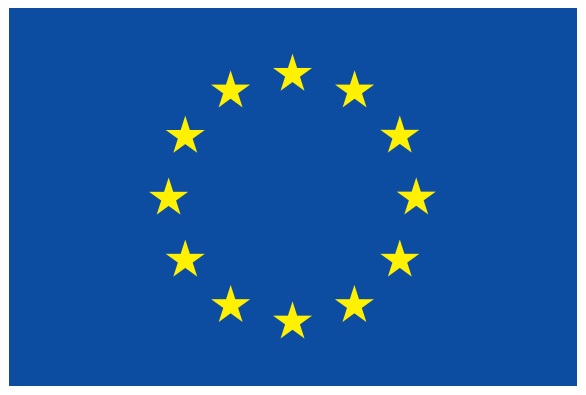}.}\\
Institute of Science and Technology Austria \\
Klosterneuburg, Austria \\
michael.anastos@ist.ac.at}

\date{}
\begin{document}

\maketitle

\begin{abstract}
We present CertifyHAM, a deterministic algorithm that takes a graph $G$ as input and either finds a Hamilton cycle of $G$ or outputs that such a cycle does not exist. If $G\sim G(n,p)$ and $p\geq \frac{100\log n}{n}$ then the expected running time of CertifyHAM is $O(\frac{n}{p})$ which is best possible. This improves upon previous results due to Gurevich and Shelah, Thomason and Alon, and Krivelevich, who proved analogous results for $p$ being constant, $p\geq 12n^{-1/3}$ and $p\geq 70n^{-1/2}$ respectively. 
\end{abstract}

\section{Introduction}

A Hamilton cycle of a graph $G$ is a cycle that visits every vertex of $G$ exactly once. The Hamilton cycle problem, which we denote by HAM, asks to determine whether a given graph $G$ has a Hamilton cycle; it is well known to be NP-complete \cite{karp1972}, and as such, there is no known polynomial time algorithm that solves it. Classic $poly(n)2^n$ time algorithms for solving HAM have been known since the early 1960’s \cite{bellman1962,held1962,kohn1977generating,ryser1963combinatorial}. Such an algorithm is the Inclusion-Exclusion HAM algorithm, described later in Section 2.  Since then 
they have been some faster algorithms that yield subexponential improvements due to Bax and Franklin \cite{bax}, Bj\"orklund \cite{bjorklund2016below}, and Bj\"orklund, Kaski and Williams \cite{bjorklund2019generalized}.
 
The problem of finding a fast algorithm that solves HAM changes drastically if you require an algorithm that is fast for ``typical" (random) graphs instead for all the graphs. Bollob\'as, Fenner and Frieze \cite{bollobasfennerfrrieze}
gave such a deterministic algorithm which we call HHAM. It takes as input the adjacency matrix of a graph $G$, runs in $O(n^{4+o(1)})$ time, and has the property that if $G\sim G(n,p)$ then,
\begin{align*}
    &\lim_{n\to \infty}\Pr( \text{HHAM finds a Hamilton cycle in }G) 
    = \lim_{n\to \infty}\Pr( G\text{ has a Hamilton cycle} ),
\end{align*}for all $p=p(n)\geq 0$.
Here, by $G(n,p)$ we denote the random binomial graph model i.e. if $G\sim G(n,p)$ then $G$ is a graph on $[n]$  where each edge in $\binom{[n]}{2}$ appears in $E(G)$ independently with probability $p$. The threshold for the existence of Hamilton cycles in $G(n,p)$ is well known. For a graph $G$ denote by $\delta(G)$ its minimum degree.  Building upon work of P\'osa \cite{posa1976} and Korshunov \cite{korshunov}, Bollob\'as \cite{bollobas1984}, Koml\'os and Szemer\'edi \cite{komlos1983} and Ajtai, Koml\'os and Szemer\'edi  \cite{ajtai1985} proved that if $G\sim G(n,p)$ where $p=\frac{\log n+\log\log n+c_n}{n}$ then,
\begin{align*}
    &\lim_{n\to \infty}\Pr(G \text{ is Hamiltonian})=
    \lim_{n\to \infty}\Pr(\delta(G) \geq 2 )
=\begin{cases}  
0 \text{ if } c_n\to -\infty.
\\ e^{-e^{-c}}  \text{ if } c_n\to -\infty.
\\ 1  \text{ if } c_n\to -\infty.
\end{cases}
\end{align*}
The result of Bollobas, Fenner and Frieze was improved in terms of running time for some values of $p$, first by Gurevich and Shelah \cite{gurevich1987}, then by Thomason \cite{thomason89} and lastly by Alon and Krivelevich \cite{alon2020}. All of the three corresponding algorithms run in $O(\frac{n}{p})$ time over the input $G\sim G(n,p)$ and take place in the adjacency matrix model. In this model, the algorithm gets the graph's adjacency matrix as an input; the algorithm can access the matrix by querying its entries. Observe that in this setting an $O(\frac{n}{p})$ running time is best possible as any algorithm w.h.p.\footnote{We say that a sequence of events $\{\mathcal{E}_n\}_{n\geq 1}$
holds {\em{with high probability}} (w.h.p.\@ in short) if $\lim_{n \to \infty}\Pr(\mathcal{E}_n)=1-o(1)$.} needs to read/query $\Omega(\frac{n}{p})$ entries of the adjacency matrix of $G$ in order to identify $n$ edges of $G$, the number of edges in a Hamilton cycle. Ferber, Krivelevich, Sudakov and Vieira  showed that $(1+o(1))n/p$ queries suffices if you allow for non-polynomial time algorithms, for $p\geq \frac{\log n+\log\log n +\omega(1)}{n}$ \cite{ferber}.

In another line of research, the Hamilton cycle problem for $G\sim G(n,p)$ is considered in the setup where $G$ is given in the form of randomly ordered adjacency lists. That is, instead of the adjacency matrix of $G$, for every vertex $v\in [n]$  we are given a random permutation of its neighbors via a list $L(v)$. In this setup, Angluin and Valiant \cite{angluin1979} gave a randomized algorithm that finds a Hamilton cycle w.h.p. in $G(n,p)$ for $p \geq \frac{C\log n}{n}$ where $C$ is a sufficiently large constant. Their result was improved with respect to $p$ first by Shamir \cite{shamir1983} and then by Bollob\'as, Fenner and Frieze \cite{bollobasfennerfrrieze} whose result is optimal with respect to $p$. Recently it was improved by Nenadov, Steger and Su with respect to the running time \cite{nenadov2020}. We summarize the above results in Table 1 given below.

\begin{table}[htbp]\label{table1}
\caption{Algorithms that solve HAM for $G\sim G(n,p)$ \em{with high probability}}
\begin{tabular}{|c|c|c|c|c|}
\hline
\textbf{ Authors} &\textbf{ Year} &\textbf{ Time} &\textbf{  p(n) }& \textbf{ Input}  \\ \hline 
 Angluin, Valiant \cite{angluin1979} & `79 & $O(n \log^2n)$ & $p\ge \frac{C_1 \log(n) }{n} $ &  adj. list \\ \hline
 Shamir \cite{shamir1983} & `83 & $ O(n^2)$ & $p\ge \frac{\log n + (3+\epsilon)\log\log n }{n} $ &  adj. list \\ \hline
 Bollobas, Fenner, Frieze \cite{bollobasfennerfrrieze} & `87 &$O( n^{4 + o(1)}) $ & $p\ge 0$ & adj. list or matrix  \\ \hline
 Gurevich, Shelah \cite{gurevich1987}& `87 &$ O(n/p) $ & $p$ const.&  adj. matrix\\ \hline
 Thomason \cite{thomason89} & `89 &$  O(n/p) $ & $p\ge 12 n^{-1/3} $ &  adj. matrix \\ \hline
 Alon, Krivelevich \cite{alon2020} & `20 &$ O(n/p)$ & $p\ge 70n^{-1/2} $ & adj. matrix \\ \hline
 Nenadov, Steger, Su \cite{nenadov2020} & `21 & $O(n)$ & $p\geq \frac{C_3\log n }{n}$ & adj. list \\ \hline
\end{tabular}

\end{table}

Another way to cope with the hardness of Hamiltonicity is by devising algorithms with a polynomial average running time. In \cite{gurevich1987}, Gurevich and Shelah gave an algorithm that determines the Hamiltonicity of an $n$-vertex graph in polynomial time on average. Equivalently, as $G(n,1/2)$ assigns the uniform measure over all graph on $n$ vertices, the expected running time of their algorithm over the input distribution $G\sim G(n,p)$ is polynomial in $n$ when $p=1/2$. This last statement raises the following question, for which values of $p$ does there exist an algorithm that solves HAM in polynomial expected running time over the input distribution $G\sim G(n,p)$. This problem is also stated as Problem 19 in the survey on Hamilton cycles in random graphs by Frieze \cite{frieze2019}. It is strictly harder than finding an algorithm that solves HAM with high probability; an algorithm with average running time can be turned into the latter by simply terminating it after a set number of steps. Gurevich and Shelah  \cite{gurevich1987}, Thomason \cite{thomason89} and Alon and Krivelevich \cite{alon2020}  gave such an algorithm  for $p\in[0,1]$ being constant, $p\geq 12n^{-1/3}$ and $p\geq 70n^{-1/2}$ respectively. respectively.

\begin{table}[htbp]
\caption{Algorithms that solve HAM for $G\sim G(n,p)$ \em{on average}}
\centering
\begin{tabular}{ |c|c|c|c|c| }
\hline
\textbf{ Authors} &\textbf{ Year} &\textbf{ Time} &\textbf{  p(n) }& \textbf{ Input}  \\ \hline 
 Gurevich, Shelah \cite{gurevich1987}& `87 &$ O(n/p) $ & $p$ const.&  adj. matrix\\ \hline
 Thomason \cite{thomason89} & `89 &$  O(n/p) $ & $p\ge 12 n^{-1/3} $ &  adj. matrix \\ \hline
 Alon, Krivelevich \cite{alon2020} & `20 &$ O(n/p)$ & $p\ge 70n^{-1/2} $ & adj. matrix \\ \hline
\end{tabular}
\end{table}

Theorem \ref{thm:main} falls into the area of average-case analysis of algorithms. For a complete survey on this area see \cite{bogdanov2006}. It is motivated by the idea that the worst-time complexity of an algorithm is usually based on some well-orchestrated instance and such instances are atypical. Thus the average time complexity may be a better measure of the performance of an algorithm. In addition, using algorithms that run fast on average is one way to cope with NP-hard problems.

A distributional problem is a pair  $(L,D)$ where $L$ is a decision problem and $D=\{D_n\}_{n\geq 1}$ where $D_n$ is a  probability distribution on the inputs of $L$ of size $n$.
We say that a distributional problem $(L,D)$ belongs to $AvgP$ if there exists an algorithm $A$ (with running time $t_A(x)$ for an instance $x$ of $L$), a constant $\epsilon$ and a polynomial $p()$ with the property that 
$$ \mathbb{E}_{x\sim D_n}(t_A^{\epsilon}(x))=O(p(n)). $$
Let $\mathcal{P}_p=\{G(n,p)\}_{n\geq 1}$. Theorem \ref{thm:main} and the upper bound on $\Pr(\delta(G)\geq 2)$ given in Lemma \ref{lem:degree}, give the following. 
\begin{corollary}\label{cor:main}
$(HAM,\mathcal{P}_p)$ belongs to $AvgP$ for all $p\geq 0$.
\end{corollary}

\subsection{Our contribution}
In this paper, we introduce CertifyHAM, an algorithm that solves HAM fast in expectation.
It implements in succession 2 algorithms that solve HAM, first the MatrixCerHAM algorithm and then the DCerHAM algorithm. MatrixCerHAM is a fast algorithm that solves HAM with high probability while DCerHAM solves HAM  slower but in expectation. This paper is devoted to presenting and analyzing DCerHAM.  The description of MatrixCerHAM and proof of Theorem \ref{thm:adjmatrix} are given in \cite{anastos2021fast}.

\begin{theorem}\label{thm:adjmatrix}
There exists a deterministic algorithm MatrixCerHAM that takes as input the adjacency matrix of a graph $G$ and outputs either a Hamilton cycle of $G$ or a certificate that $G$ is not Hamiltonian or FAILURE. If $G\sim G(n,p)$ for some $p=p(n)\geq 0$ then with probability $1-o(n^{-7})$ it outputs either a Hamilton cycle of $G$ or a certificate that $G$ is not Hamiltonian by making at most $\frac{n}{p}+o(\frac{n}{p\log\log n})$ queries and running in $O(\frac{n}{p})$ time.
\end{theorem}

A $2$-matching of a graph $G$ is a set of edges $M\subseteq E(G)$ such that every vertex in $G$ is incident to at most two edges in $M$. We say that a 2-matching $M$ saturates a set of vertices $S\subset V(G)$ if every vertex in $S$ is incident to exactly two edges in $M$. Thus a necessary condition for $G$ to be Hamiltonian is that for every $S\subseteq V(G)$ there exists some 2-matching $M_S$ that saturates $S$. The certificate, described at Theorem \ref{thm:adjmatrix},  consists of a set of vertices that violates this necessary condition.

\begin{theorem}\label{thm:main}
There exists a deterministic algorithm DCerHAM that takes as input the adjacency matrix of a graph $G$ and outputs either a Hamilton cycle of $G$ or a certificate that $G$ is not Hamiltonian. If $G\sim G(n,p)$, $p\geq \frac{5000}{n}$ then the running time $T$ of DCerHAM with input $G$ can be decomposed into two random variables, $T_A$ and $T_B$ where $T_A$ is bounded above by $O(n^7)$ and $T_B$ has expected value $O(n)$.
\end{theorem}
\begin{theorem}\label{thm:mainmain}
There exists a deterministic algorithm CertifyHAM that takes as input the adjacency matrix of a graph $G$ and outputs either a Hamilton cycle of $G$ or a certificate that $G$ is not Hamiltonian. If $G\sim G(n,p)$, $p\geq \frac{5000}{n}$ then its expected running time is $O(\frac{n}{p})$.
\end{theorem}
\begin{remark}
In \cite{anastos2021fast} along MatrixCerHAM, the algorithm ListCerHAM is presented. ListCerHAM is an algorithm that solves HAM in $O(n)$ time with probability $1-o(n^{-7})$ when the input $G$ is drawn according to $G(n,p)$ for some $p=p(n)\geq 0$ in the adjacency list model. By first running ListCerHAM and then, in case of a FAILURE, generating the adjacency matrix of $G$ in $O(n^2)$ time and finally running the DCerHAM algorithm ones solves HAM in $O(n)$ expected running time  in the adjacency lists model.
\end{remark}
\textbf{Proof of Theorem \ref{thm:mainmain}:}
CertifyHAM algorithm: Run MatrixCerHAM. If it returns FAILURE, then run DCerHAM.  By theorems \ref{thm:adjmatrix} and \ref{thm:main}, CertifyHAM is a deterministic algorithm that solves HAM. Its expected running time is bounded above by the Running time of MatrixCerHAM  $+\Pr($MatrixCerHAM returns FAILURE) $\times$ $T_A$+$T_B$.
By theorems \ref{thm:adjmatrix} and \ref{thm:main} this is equal to $O(\frac{n}{p})+o(n^{-7})\cdot O(n^7)+O(n)=O(\frac{n}{p}).$
\qed

The highlight of this paper is Theorem \ref{thm:main} since for any $ p\leq 70n^{-0.5}$ not even a subexponential expected running time algorithm for solving HAM was known. On the other hand Theorem \ref{thm:main} is valid for values of $p$ pass the threshold for Hamiltonicity, which may be consider as a natural barrier for this problem. 

In all of the above theorems, we assume that we work in the adjacency matrix model. Thus the algorithms are given $A_G$, the adjacency matrix of $G$, and each entry of $A_G$ can be accessed in $O(1)$ time. Recall that in this model any algorithm that solves HAM has to identify at least $n$ edges of $G$; hence w.h.p. it has to read at least $(1+o(1))n/p$ entries of the matrix $A_G$. Thus the running times of both MatrixCerHAM and CertifyHAM are optimal, this is not true for DCerHAM. In addition, MatrixCerHAM makes asymptotically the minimum number of queries possible, thus matching the result of  Ferber, Krivelevich, Sudakov and Vieira in this aspect. It is also worth highlighting that all of MatrixCerHAM, DCerHAM and CertifyHAM are deterministic, as opposed for example to the algorithm of  Nenadov, Steger and Su \cite{nenadov2020} which is a randomized one. 

DCerHAM arises from derandomizing RCerHAM, a randomized algorithm that solves HAM. Its performance is given by Theorem \ref{thm:Rmain}. The D and R in the names DCerHAM and RCerHAM stand for deterministic and randomized respectively. Later on, we also introduce the CerHAM algorithm. CerHAM is the main subroutine of both DCerHAM and RCerHAM.

\begin{theorem}\label{thm:Rmain}
There exists a randomized algorithm RCerHAM that takes as input the adjacency matrix of a graph $G$ and outputs either a Hamilton cycle of $G$ or a certificate that $G$ is not Hamiltonian. If $G\sim G(n,p)$, $p\geq \frac{100\log n}{n}$ then the running time $T$ of RCerHAM with input $G$ can be decomposed into two random variables, $T_A$ and $T_B$ where $T_A$ is bounded above by $O(n^7)$ and $T_B$ has expected value $O(n)$.
\end{theorem}

\subsection{The RCerHAM  algorithm}
In this subsection, we give a sketch of the RCerHAM algorithm. For its description, we introduce the definition of a full path packing and the P\'osa rotations procedure. We then present the main elements of its analysis. We also describe how to derandomize it to obtain DCerHAM.

RCerHAM first generates a subset of $E(G)$, $F_0$. To do so, it generates $G'\sim G(n,1/2)$ and then queries the edges of $G$ that belong to $G'$. By querying an edge of $G$ we refer to reading the corresponding entry of $A_G$ and revealing whether it is filled with 1 or 0; equivalently, revealing whether the corresponding edge belongs to $G$ or not. Then it lets $F_0=E(G)\cap E(G')$. It proceeds by executing the CerHAM algorithm, stated in Section 3, with input $G$ (as before, the adjacency matrix of $G$ is given to CerHAM) and $F_0$. Initially CerHAM sets $G_0=([n],F_0)$.

Given a path $P=v_1,v_2,....,v_k$ and $v_iv_k$, $1\leq i<k$ we say that the path $P'=v_1,v_2,...,v_i,v_k,v_{k-1},$ $...,v_{i+1}$ is obtained from $P$  via a P\'osa rotation that fixes $v_1$. For a path $P$ we say the edge $e$ is a {\em{booster}} for $P$ if $\{e\}\cup P$ is either a longer path or a cycle. Many algorithms and heuristics use P\'osa rotations to grow a path of a given graph $G$ (by finding boosters along the way) first into a Hamilton path and then into a Hamilton cycle of $G$. For example, they can be utilized in the following manner. Let $P$ be a $u,v$ path in some graph $Q$. Starting from $P$ perform all sequences of P\'osa rotations of length at most $n$ that fix $v$. Let $\mathcal{P}$ be the set of paths obtained and $End$ be the set of the corresponding endpoints. Now for $u\in End$ let $P_{uv}$ be a $v$-$u$ path in $\mathcal{P}$ and do the following. Starting from $P_{uv}$ perform all sequence of P\'osa rotations of length at most $n$ that fix $u$, let $\mathcal{P}_u$ be the set of paths obtained and $End_u$ be the set of the corresponding endpoints. Observe that for $u\in End$ and $w\in End_u$ the edge $uw$ is a booster of the corresponding path $P_{uw}$ in $\mathcal{P}_u$. In addition every edge from $\{u,w\}$ to $[n]\setminus V(P_{uw})$ is a booster of $P_{uw}$. Then one searches for a booster in the underlying graph $Q$ to get either a longest one or a cycle $C$. If $C$ is not Hamiltonian, it can be opened up and extended via any edge from $V(C)$ to $[n]\setminus V(C)$.

Our algorithm does not grow a single path. Instead, it grows a set of vertex disjoint paths which eventually merge into a single one. We describe this as reducing the size of a full path packing. A full path packing (FPP for shorthand) of $G$ is a set of vertex disjoint paths in $G$ that cover $V(G)$. Here we consider single vertices to be paths of length 0. Thus $G$ has at least one trivial FPP, namely the one that consists of the $n$ paths of length 0 i.e. $[n]$. In addition, a Hamilton path corresponds to an FPP of size 1. For convenience, we refer to a set consisting of a single Hamilton cycle as an FPP of size 0. CerHAM is an iterative deterministic algorithm. It maintains a full path packing $\cP$ of $G$, starting with the trivial one, and a set $S$ with the property $|N_{G_0}(S)|<2|S|$, starting with the empty set (we refer to sets with this property as non-expanding). For a graph $G$ and $S\subseteq V(G)$ we denote by $N_{G_0}(S)$ the set of neighbors of the vertices in $S$ in the set $V(G)\setminus S$. At each iteration, using P\'osa rotations, CerHAM attempts to find an FPP of smaller size, eventually identifying one of size 0 which corresponds to a Hamilton cycle of $G$.

CerHAM runs a while loop until it solves HAM or identifies a very unlikely event, i.e. one that occurs with probability $O(2^{-n})$. In the second case, it exits the loop and implements the Inclusion-Exclusion HAM algorithm, described in Section 2. It starts each iteration with an FPP $\cP$, a non-expanding set $S$ and a set $F$ that consists of all the edges of $G$ that have been queried so far. It lets $G_R$ be the graph whose edges are those of $G$ that have been queried so far, i.e. $E(G_R)=F\cap E(G)$. Then it adds to $G_R$ a set of edges $J$ of size $|\cP|-1$ that join the paths in $\cP$ into a Hamilton path in $G_R\cup J$. It continues by performing sequences of P\'osa rotations and identifies sets $End$, $\{End_u\}_{u\in End}$. It does not perform all the possible P\'osa rotations; instead, it only considers P\'osa rotations for which the endpoints of the obtained paths do not belong to the non-expanding set $S$; the reason will become clear later on. The sets $End$, $\{End_u\}_{u\in End}$ have the property that for every $u\in End$ and $w\in End_u$ there exists a Hamilton path $H_{uw}$ in $G_R\cup J$ from $u$ to $w$ that has at most $|\cP|-2$ edges from $J$, thus the edge $uw$ is a booster for some Hamiltonian path. In addition for $T\in \{End\}\cup\{End_u: u\in End\}$ we will have that the set $T$ is non-expanding. Now either (I) all of the sets $ End,\{ End_u\}_{u\in End}$ are sufficiently large (have size at least $0.25n$) or (II) at least one of them is small. In the first case, it identifies and queries a booster $uw$ with $u\in End$ and $w\in End_u$ that has not been queried yet. If $wu\in E(G)$ then it adds $wu$ to the Hamilton path from $u$ to $w$ to create a Hamilton cycle $H$ that uses at most $|\cP|-2$ edges from $J$. It then considers $H\setminus J$ and updates $\cP$ to a full path packing of a smaller size or claims that $G$ is Hamiltonian; in such a case, we say it improves on $|\cP|$ (see Figure \ref{fig1}).

\begin{figure}[htbp]
\begin{center}
    
\begin{tikzpicture}
\node at (-0.2,0.4) {(a)};
\draw  (0,0)--(2,0);
\draw  (2.5,0) circle [radius=.1];
\draw  (3,0) circle [radius=.1];
\draw  (3.5,0)--(4.5,0);
\draw  (5,0) circle [radius=.1];
\draw  (5.5,0)--(7,0);

\draw (0,0) circle [radius=.1];
\draw (0.5,0) circle [radius=.1];
\draw  [fill=red](1,0) circle [radius=.1];
\draw  [fill=red](1.5,0) circle [radius=.1];
\draw (2,0) circle [radius=.1];
\draw (2.5,0) circle [radius=.1];
\draw (3,0) circle [radius=.1];
\draw (3.5,0) circle [radius=.1];
\draw (4,0) circle [radius=.1];
\draw (4.5,0) circle [radius=.1];
\draw (5,0) circle [radius=.1];
\draw (5.5,0) circle [radius=.1];
\draw  [fill=red](6,0) circle [radius=.1];
\draw (6.5,0) circle [radius=.1];
\draw (7,0) circle [radius=.1];

\node at (-0.2,-1.1) {(b)};
\draw  (0,-1.5)--(2,-1.5);
\draw  (2.5,-1.5) circle [radius=.1];
\draw  (3,-1.5) circle [radius=.1];
\draw  (3.5,-1.5)--(4.5,-1.5);
\draw  (5,-1.5) circle [radius=.1];
\draw  (5.5,-1.5)--(7,-1.5);
\draw[dashed]  (2,-1.5)--(3.5,-1.5);
\draw[dashed]  (4.5,-1.5)--(7,-1.5);

\draw (0,-1.5) circle [radius=.1];
\draw (0.5,-1.5) circle [radius=.1];
\draw  [fill=red](1,-1.5) circle [radius=.1];
\draw  [fill=red](1.5,-1.5) circle [radius=.1];
\draw (2,-1.5) circle [radius=.1];
\draw (2.5,-1.5) circle [radius=.1];
\draw (3,-1.5) circle [radius=.1];
\draw (3.5,-1.5) circle [radius=.1];
\draw (4,-1.5) circle [radius=.1];
\draw (4.5,-1.5) circle [radius=.1];
\draw (5,-1.5) circle [radius=.1];
\draw (5.5,-1.5) circle [radius=.1];
\draw  [fill=red](6,-1.5) circle [radius=.1];
\draw (6.5,-1.5) circle [radius=.1];
\draw (7,-1.5) circle [radius=.1];

\node at (-0.2,-2.6) {(c)};
\draw  (0,-3)--(2,-3);
\draw  (2.5,-3) circle [radius=.1];
\draw  (3,-3) circle [radius=.1];
\draw  (3.5,-3)--(4.5,-3);
\draw  (5,-3) circle [radius=.1];
\draw  (5.5,-3)--(7,-3);
\draw[dashed]  (2,-3)--(3.5,-3);
\draw[dashed]  (4.5,-3)--(7,-3);

\draw (0,-3) circle [radius=.1];
\draw (0.5,-3) circle [radius=.1];
\draw  [fill=red](1,-3) circle [radius=.1];
\draw  [fill=red](1.5,-3) circle [radius=.1];
\draw (2,-3) circle [radius=.1];
\draw (2.5,-3) circle [radius=.1];
\draw (3,-3) circle [radius=.1];
\draw (3.5,-3) circle [radius=.1];
\draw (4,-3) circle [radius=.1];
\draw (4.5,-3) circle [radius=.1];
\draw (5,-3) circle [radius=.1];
\draw (5.5,-3) circle [radius=.1];
\draw  [fill=red](6,-3) circle [radius=.1];
\draw (6.5,-3) circle [radius=.1];
\draw (7,-3) circle [radius=.1];

\node at (7.2,-3.3) {u};
\node at (4,-3.3) {x};
\node at (4.5,-3.3) {y};
\node at (5.5,-3.3) {w};
\node at (6,-3.3) {z};

\draw [thick,dotted] (4,-3) to [out=50,in=130] (7,-3);

\node at (-0.2,-4.1) {(d)};
\draw  (0,-4.5)--(2,-4.5);
\draw  (2.5,-4.5) circle [radius=.1];
\draw  (3,-4.5) circle [radius=.1];
\draw  (3.5,-4.5)--(6,-4.5);
\draw  (5,-4.5) circle [radius=.1];
\draw[dashed]  (2,-4.5)--(3.5,-4.5);
\draw[dashed]  (6,-4.5)--(7,-4.5);

\draw (0,-4.5) circle [radius=.1];
\draw (0.5,-4.5) circle [radius=.1];
\draw  [fill=red](1,-4.5) circle [radius=.1];
\draw  [fill=red](1.5,-4.5) circle [radius=.1];
\draw (2,-4.5) circle [radius=.1];
\draw (2.5,-4.5) circle [radius=.1];
\draw (3,-4.5) circle [radius=.1];
\draw (3.5,-4.5) circle [radius=.1];
\draw (4,-4.5) circle [radius=.1];
\draw (4.5,-4.5) circle [radius=.1];
\draw (5,-4.5) circle [radius=.1];
\draw [fill=red] (5.5,-4.5) circle [radius=.1];
\draw (6,-4.5) circle [radius=.1];
\draw (6.5,-4.5) circle [radius=.1];
\draw (7,-4.5) circle [radius=.1];

\node at (7.2,-4.8) {y};
\node at (4,-4.8) {x};
\node at (4.5,-4.8) {u};
\node at (5.5,-4.8) {z};
\node at (6,-4.8) {w};
\node at (0,-4.8) {v};

\end{tikzpicture}
\caption{(a) An FPP of size $6$, the vertices in $S$ are in red. (b) Adding  $5$ edges (dashed) creates a Hamilton path. (c) Performing a P\'osa using $ux$ results in the path in (d). The edge $uw$, even if present will not be used to perform a P\'osa rotations as such a rotation will result in a path with endpoint $z\in S$. (d) If present, adding the edge $vy$ and removing the dashed edges gives an FPP of size $5$.}
\label{fig1}
\end{center}
\end{figure}
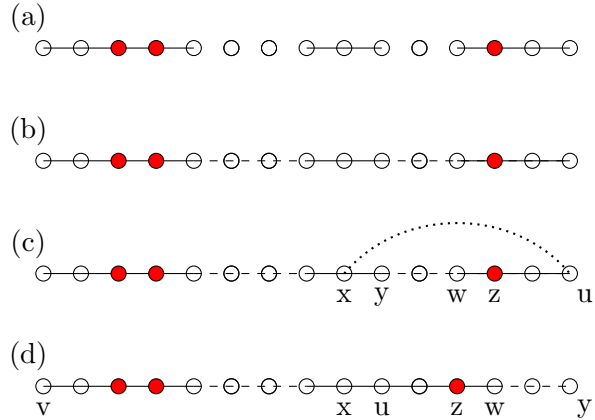

Thus if (II) never occurs at each iteration, our algorithm will improve on $\cP$ with some positive probability and eventually either make too many queries, this is a very unlikely event, or construct a Hamilton cycle. The more convoluted part of our algorithm is the one that deals with the moderately unlike event (II). Broadly speaking, this is where the difficulty of developing fast in expectation algorithms that solve HAM lies; in dealing with events that are moderately unlike and which make it challenging but not impossible to find Hamilton cycles. 

We say that a path $P$ covers the vertex $v$ if $v\in V(P)$ and $v$ is not an endpoint of $P$. In extension we say that a set of paths $\mathcal{P}'$ covers $U$ if for every $v\in U$ there exists $P\in \mathcal{P}'$ such that $P$ covers $v$. If (II) occurs, then as all of $End$, $\{End_u\}_{u\in End}$, are non-expanding, a small non-expanding set $T$ is identified. As both $T$ and $S$ are non-expanding, the set $S\cup T$ is non-expanding. CerHAM  then updates $S=S\cup T$  (we will show that this event occurs with probability $n^{-\Omega(|S|)}$). In the very unlikely event that $S$ is of medium size, say $|S|>2^{\Omega(n/\log n)}$, we exit the while loop. If it does not, the algorithm adds $T$ to $S$ (i.e it replaces $S$ with $S\cup T$) and makes a brute force attempt to find a set of vertex disjoint paths $\mathcal{P}'$ that cover  $S$. If it is unsuccessful, it certifies that $G$ is not Hamiltonian. Else it adjusts  $\mathcal{P}$ such that it contains  $\mathcal{P}'$ (see Figure \ref{fig2}. It then proceeds with the next iteration. Observe that since at the P\'osa rotations part of the algorithm we avoid the ones that result in an endpoint in $S$, before constructing $\mathcal{P}'$ the sets $T$ and $S$ disjoint. Thus at every iteration that (II) occurs, $S$ grows until it becomes of medium size; this implies that (II) occurs at most $n$ and CerHAM eventually terminates.

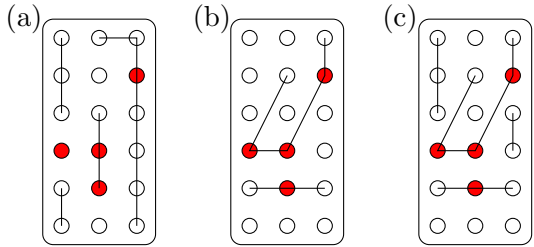
\begin{figure}[htbp]
\begin{center}
\begin{tikzpicture}

\node at (-0.5,2.75) {(a)};

\draw (0,0) circle [radius=.1];
\draw (0,0.5) circle [radius=.1];
\draw[fill=red] (0,1) circle [radius=.1];
\draw (0,1.5) circle [radius=.1];
\draw (0,2) circle [radius=.1];
\draw (0,2.5) circle [radius=.1];

\draw (0.5,0) circle [radius=.1];
\draw [fill=red](0.5,0.5) circle [radius=.1];
\draw [fill=red](0.5,1) circle [radius=.1];
\draw (0.5,1.5) circle [radius=.1];
\draw (0.5,2) circle [radius=.1];
\draw (0.5,2.5) circle [radius=.1]; 

\draw (1,0) circle [radius=.1];
\draw (1,0.5) circle [radius=.1];
\draw (1,1) circle [radius=.1];
\draw (1,1.5) circle [radius=.1];
\draw [fill=red] (1,2) circle [radius=.1];
\draw (1,2.5) circle [radius=.1]; 

\draw  (0,0)--(0,0.5);
\draw  (0,1.5)--(0,2.5);
\draw  (0.5,0.5)--(0.5,1.5);
\draw  (1,0)--(1,2.5)--(0.5,2.5);
\draw[rounded corners] (-0.25, -0.25) rectangle (1.25, 2.75) {};

\node at (2,2.75) {(b)};

\draw (2.5,0) circle [radius=.1];
\draw (2.5,0.5) circle [radius=.1];
\draw[fill=red] (2.5,1) circle [radius=.1];
\draw (2.5,1.5) circle [radius=.1];
\draw (2.5,2) circle [radius=.1];
\draw (2.5,2.5) circle [radius=.1];

\draw (3,0) circle [radius=.1];
\draw [fill=red](3,0.5) circle [radius=.1];
\draw [fill=red](3,1) circle [radius=.1];
\draw (3,1.5) circle [radius=.1];
\draw (3,2) circle [radius=.1];
\draw (3,2.5) circle [radius=.1]; 

\draw (3.5,0) circle [radius=.1];
\draw (3.5,0.5) circle [radius=.1];
\draw (3.5,1) circle [radius=.1];
\draw (3.5,1.5) circle [radius=.1];
\draw [fill=red] (3.5,2) circle [radius=.1];
\draw (3.5,2.5) circle [radius=.1]; 

\draw  (2.5,0.5)--(3.5,0.5);
\draw  (3,2)--(2.5,1)--(3,1)--(3.5,2)--(3.5,2.5);
\draw[rounded corners] (2.25, -0.25) rectangle (3.75, 2.75) {};

\node at (4.5,2.75) {(c)};

\draw (5,0) circle [radius=.1];
\draw (5,0.5) circle [radius=.1];
\draw[fill=red] (5,1) circle [radius=.1];
\draw (5,1.5) circle [radius=.1];
\draw (5,2) circle [radius=.1];
\draw (5,2.5) circle [radius=.1];

\draw (5.5,0) circle [radius=.1];
\draw [fill=red](5.5,0.5) circle [radius=.1];
\draw [fill=red](5.5,1) circle [radius=.1];
\draw (5.5,1.5) circle [radius=.1];
\draw (5.5,2) circle [radius=.1];
\draw (5.5,2.5) circle [radius=.1]; 

\draw (6,0) circle [radius=.1];
\draw (6,0.5) circle [radius=.1];
\draw (6,1) circle [radius=.1];
\draw (6,1.5) circle [radius=.1];
\draw [fill=red] (6,2) circle [radius=.1];
\draw (6,2.5) circle [radius=.1]; 

\draw  (5,0.5)--(6,0.5);
\draw  (5.5,2)--(5,1)--(5.5,1)--(6,2)--(6,2.5);
\draw  (5,2.5)--(5,1.5);
\draw  (6,1)--(6,1.5);

\draw[rounded corners] (4.75, -0.25) rectangle (6.25, 2.75) {};

\end{tikzpicture}
\end{center}
\caption{(a) $\mathcal{P}$-an FPP of size $7$, the set $S$ in red. (b) A set of paths $\mathcal{P'}$ that covers $S$. (c) Removing from (a) the edges incident to  $\mathcal{P'}$ and adding to it the edges in $(b)$ gives a new FPP of size $9$ that covers $S$. 
}
\label{fig2}
\end{figure}

One of the interesting features of our algorithm is that it allows the size of $\cP$ to increase. Indeed, every time (II) occurs $|\mathcal{P}|$ may be increased by at most $2|S|$; this is the price that we pay to ensure that $\mathcal{P}'\subset \mathcal{P}$. Let $\Delta$ be the sum over the iterations at which (II) occurs of the increase of $|\mathcal{P}|$. The above imply that $\Delta\leq 2n^2$ and one may have to wait until $2n^2+n$ distinct iterations where we improve on $|\mathcal{P}|$ occur until a Hamilton cycle is obtained, this is not good enough. To obtain an $O(n)$ bound on $\Delta$  we implement the following additional subroutine whenever (II) occurs. After setting $S=S\cup T$ we look for non-expanding sets $W \subset [n]\setminus S$ of size at most $|S|$ until no such a set exists. As soon as we identify such a set, we add it to $S$ and then continue our search. Once our search is finished, the set $S$ has the property that no set $W\subset [n]\setminus S$ is non-expanding. The next time (II) occurs a non-expanding set $W=T$ is identified, thus $|T|>|S|$. It follows that after each time (II) occurs the set $S$ doubles in size, thus $\Delta \leq  \sum_{i=1}^{\log_2 n}2\cdot 2^i\leq 4n$. Initially, $|\cP|=n$ and $|\cP|=0$ implies that we have identified a Hamilton cycle. Each time an attempt to improve $|\cP|$ is made $|\cP|$ is decreased by at least $1$ with probability $0.5p$. As $\Delta\leq 4n$ we may improve on $|\cP|$ at most $5n$ times until it reaches  $0$. Thus the probability $0.01n^2$ attempts to improve $\cP$ are made is $\Pr(Bin(0.01n^2,0.5p)<5n)$ and the corresponding event is very unlikely, causing CerHAM to exit the while-loop. This completes the description of RCerHAM.

The single point in which DCerHAM and RCerHAM differ is in generating $G_0$. Later on, at the execution of DCerHAM we derandomize this step by letting $F_0$ be the edge set of $H_n$ and defining $G_0$ by $V(G_0)=[n]$ and $E(G_0)=F_0\cap E(G)$. $H_n$ will be a well-chosen pseudorandom graph of linear in $n$ minimum degree containing a constant fraction of the edges of $K_n$. Roughly speaking, a graph $H$ is pseudorandom if the edge counts between any two linear sets is close to what one would expect it to be in $G\sim G(n,p)$ with $p=E(H)/\binom{n}{2}$. Substituting $G(n,0.5)$ with $H_n$ will suffices for our calculations to go through.

\subsubsection{Comparison to previous work}
Our algorithm does not build upon the previous algorithms of Gurevich and Shelah, Thomason and Alon and Krivelevich. For example, the algorithm given by Alon and Krivelevich first builds a cycle that covers all the vertices of degree at most $np/2$ and then grows this cycle by sequentially absorbing the rest of the vertices using local alterations. For example given a cycle $C$ of $G$ on $|V(G)|-1=n-1$ vertices and  a vertex $v$ outside $C$ their algorithm may look for vertices $a,b,c,d$ that appear in this order on the cycle such that if we let $a^-,b^-,c^-,d^-$ be the vertices preceding them on $C$ then $a,d$ are neighbors of $v$ and $a^-b,b^-c,c^-d^-\in E(G)$. In this case $(E(c)\cup \{a^-b,b^-c,c^-d^-,av,vd\})\setminus \{aa^-,bb^-,cc^-,dd^-\}$ is a Hamilton cycle. Their analysis is based on the fact that with probability $1-e^{-\Omega(n)}$ there are at most $np/20$ vertices of degree smaller than $np/2$. This statement is not true for $p=o(n^{-0.5)}$. Thus the vertex $v$ has at least $np/10$ neighbors $a,d$ such that each of $a^-,d^-$ has  at least $np/2$ neighbors $b$ and $c^-$. This gives $\Omega(n^2)$ quadruples $\{a,b,c,d\}$ such that for each quadruple there exists a unique edge $e$ which can be used to absorb $v$ into $C$. On the other hand, the bottleneck of our algorithm is found in the fact that the statement ``the probability there exists a set $S$ of size $m \leq 0.25n$ such that $|N_{G_0}(S)|<2|S|$, where $G_0\sim G(n,0.5p)$, is at most $n^{-1.1m}$" is false for $p$ significantly smaller than $\frac{100\log n}{n}$. 

\subsection{Organization of the paper} 
The rest of the paper is organized as follows. In Section 2, we present a number of subroutines implemented by CerHAM. We give the formal description of both RCerHAM and CerHAM in Section 3. In Section 4 we derandomize RCerHAM and prove Theorem \ref{thm:main}. For ease of the presentation of the paper we assume that $n$ is significantly large for certain arguments to hold and we omit ceilings and floors. Throughout the paper we assume that $n$ is significantly large so that various inequalities hold.
\section{Preliminaries}
\subsection{The Inclusion-Exclusion HAM algorithm}
For a graph $G$, $u,v\in V(G)$, $S\subset V(G)$ and $e\in E(G)$ denote by $G-S$ the subgraph of $G$ induced by $V(G)\setminus S$, by $G-e$ the graph obtained by removing $e$ from $G$, by $A_G$ the adjacent matrix of $G$ and finally by $A_G(v,u)$ the entry of $A_G$ that corresponds to the pair of vertices $v,u$. The Inclusion-Exclusion HAM algorithm is based on the observation that  $A_G^{n-1}(u,v)$ equals  the number of walks from $v$ to $u$ of length $n-1$ and that the number of Hamiltonian paths from $v$ to $u$ can be obtained from the Inclusion-Exclusion principle and the matrices $\{A_{(G-uv)-S}^{n-1}\}_{S\subseteq V(G)\setminus \{v,u\}}$.

More concretely, for each $e\in E(G)$ the Inclusion-Exclusion HAM algorithm executes the following steps. It lets $F=G-e$, $e=\{v,u\}$ and calculates,
$$H(v,u)= \sum_{k=0}^{n-2} (-1)^k\sum_{S \in \binom{ V(F)\setminus\{v,u\}}{k}} A_{F-S}^{n-1}(v,u).$$
If $H(v,u)\neq 0$ it then outputs that $G$ has a Hamilton cycle that passes through the edge $e=\{v,u\}$.
If $H(v,u)= 0$ for every $\{u,v\}\in E(G)$ then it outputs ``$G$ is not Hamiltonian".

For the correctness of the algorithm observe that $A_{F-S}^{n-1}(v,u)$ equals the number of walks from $v$ to $u$ of length $n-1$ that do not pass through $S$. Thus, by the inclusion-exclusion principle, $H(v,u)$ equals to the number of Hamiltonian paths of $G$ from $v$ to $u$. Finally as for each edge at most $n2^n$ matrix  multiplications are performed its worst-time running complexity is $O(|E(G)|\cdot n^3 \cdot n2^n)=O(n^62^n)$.

\subsection{Restricted P\'osa Rotations}
Given a Hamilton path $P=v_1,v_2,....,v_n$ and $v_iv_n$, $1\leq i<n$ we say that the path $P'=v_1,v_2,...,v_i,v_n,v_{n-1},...,v_{i+1}$ is obtained from $P$  via a P\'osa rotation that fixes $v_1$. We call $v_i$ and $v_{i+1}$ the pivot vertex and the new endpoint respectively. We call $v_nv_i$ and $v_iv_{i+1}$ the inserted and the deleted edge respectively. We will not perform all P\'osa rotations possible. Instead, we will perform only the P\'osa rotations for which the pivot vertex does not belong to some prescribed set $S$ and the new endpoint is not an endpoint of some other path that has already been identified. The exact procedure, which we call RPR, for \emph{Restricted P\'osa Rotations}, follows shortly. It takes as an input a graph $G$, a Hamilton path $P$ of $G$, an endpoint $v$ of $P$ and a set $S$ not containing the endpoints of $P$. It then outputs  sets $End$ and $\cP$ such that for each $w\in End$ there exists a Hamilton path in $\cP$ from $v$ to $w$. At the description of RPR $Q$ is a queue.  
\begin{algorithm}[H]
\caption{$RPR(G,P,v,S)$}
\begin{algorithmic}[1]
\\Let $u$ be such that $u,v$ are the endpoints of $P$.
\\ $\cP=\{P\}$,  $End=\{u\}$ and $Q$ be a queue, currently containing only $u$. 
\While{$Q \neq \emptyset$}
\\\hspace{5mm} Let $w$ be the first vertex in $Q$ and $P_w$ be the unique Hamilton path in $\cP$ from $v$ to $w$.
 \For{ $z\in N(w)\setminus S$}
\If{ the P\'osa rotation starting from $P_w$ that fixes $v$ and the inserted edge is $wz$ results to an endpoint which does not belong to $End$}
\\\hspace{15mm} Perform the corresponding P\'osa rotation, add the new endpoint to $End$ and to the end of $Q$ and the corresponding path to $\cP$.
\EndIf 
\EndFor
\\\hspace{5mm} Remove $w$ from $Q$.
\EndWhile
\\ Output $End, \cP$.
\end{algorithmic}
\end{algorithm} 
As every vertex enters at most once the set $Q$, both of the output sets $End$ and $\cP$ have size at most $n$ and the RPR algorithm runs in $O(n^2)$ time. We now present the key lemma related to the RPR algorithm.
\begin{lemma}\label{lem:posa}
Let $G$ be a graph, $S\subset V(G)$, $P$ a Hamiltonian path of $G$ with endpoints $u,v$ and $End, \cP =$RPR$(G,P,v,S)$. Then,
$$|N_G(End)\setminus S| <2|End|.$$
\end{lemma}
\begin{proof}
Let $u=u_1,u_2,....,u_\ell$ be the vertices in $End$ in the order that are added to $Q$ and $P_{u_1},P_{u_2},$ $...,P_{u_\ell}$ be the corresponding Hamiltonian paths. Let $T_1=\{u_1,r_{1,1}\}$ and $T_i=\{u_i,r_{i,1},r_{i,2}\}$ for $2\leq i\leq \ell$,  where $r_{i,2}u_i$ is the deleted edge at the P\'osa rotation that results to the Hamilton path $P_{u_i}$ for $2\leq i\leq \ell$ and $r_{i,1}$ is the vertex preceding $u_i$ on $P_{u_i}$  for $1\leq i\leq \ell$.

Let $w \in \cup_{i\in \ell }N(u_i)$ be such that $w\notin S$ and $k$ be minimum such that $w\in N(u_k)$. Let $x$ be the neighbor of $w$ between $w$ and $u_k$ (included) on the path $P_k$ (so $xw\in E(P_k)$). Then either $x=u_k$ or, due to the minimality of $k$, $x\notin \{u_1,u_2,...,u_k\}$. In the first case $w=r_{k,1}$. In the second case, during the $k^{th}$ execution of the while loop of RPR a P\'osa rotation is performed where the inserted edge is the edge $u_kw$ and the new endpoint is $x$. Therefore, in the second case, $x=u_{k'}$ for some $k'>k$ and $w=r_{k',2}$. Thus $N_G(End)\setminus S \subseteq \cup_{i\leq \ell}T_\ell$ and 
$$|N_G(End)\setminus S|\leq 1+2(|End|-1)<2|End|.$$
\end{proof}
The key advantage of performing restricted P\'osa rotations as opposed to normal P\'osa rotations is given by the following observation.

\begin{observation}\label{lem:posaendpoints}
If $S=W\cup N(W)$ and $End,\cP=$RPR$(G,P,v,S)$ then $End$ and $W$ are disjoint. Indeed, let $u,v$ be the endpoints of $P$. Then, $u\notin S$. Thereafter for a vertex $w\in [n]$ only a P\'osa rotation with a pivot vertex in $N(w)\setminus S$ may result to a path with $w$ as an endpoint. As $N(w)\setminus S=\emptyset$ for $w\in W$ we have that no vertex in $W$ ever enters $End$. 
\end{observation}

We use RPR as a subroutine of the ReducePaths algorithm stated below. ReducePaths takes as input a graph $G$, an FPP $\cP$ of $G$ and $S\subset V(G)$. It outputs sets $End$, $\{End_u\}_{u\in End}$ and a set of FPPs $\cU$ with the following property. For $u\in End$ and $w\in End_u$ there exists an FPP $\cP_{u,w} \in \cU$  of size at most $|\cP|$ with the property that adding $uw$ to $\cP_{u,w}$ either creates a Hamilton cycle, if $|\cP_{u,w}|=1$,  or joins two paths in $\cP_{u,w}$ and creates a full path packing of size at most  $|\cP|-1$. ReducePaths$(G,\cP,S)$ calls RPR $O(n)$ times, hence its running time is $O(n^3)$.  

\begin{algorithm}[H]
\caption{$ReducePaths(G,\cP,S)$}
\begin{algorithmic}[1]
\\ Let $\cP=\{P_1,P_2,...,P_k\}$. Let $v_i,u_i$  be the endpoints of $P_i$, $1\leq i\leq k$.
\\ Let $R=\{u_iv_{i+1}:1\leq i \leq k-1\}$ and $P=v_1P_1u_1v_2P_2u_2,...,u_{k-1}v_{k}P_{k}u_k$. 
\\  $End, \cP =$RPR$(G\cup R, P, v_1,S)$
\For{$u\in End$} 
\\ \hspace{5mm} Let $P_u$ be the Hamilton path from $v$ to $u$ in $\cP$. 
\\ \hspace{5mm} $End_u, \cP_u =$RPR$(G\cup R, P_u, u,S)$.
\\ \hspace{5mm} For $w\in End_u$ let $P_{u,w} \in \cP_u$ be the $u$-$w$ Hamilton path in $G\cup R$ and let $\cP_{u,w}$ be the full path packing obtained from $P_{u,w}$ by removing the edges in $R$.
\EndFor
\\ Set $\cU=\{\cP_{u,w}:u\in End, w\in End_w\}$
\\ Output $End,\{End_u\}_{u\in End}, \cU$.
\end{algorithmic}
\end{algorithm} 

\subsection{Identifying non-expanding sets}
The FindSparse algorithm takes as  input a graph $G$, a subgraph $G_R$ of $G$, a set of edges $F\subset \binom{n}{2}$ and a subset $S$ of $[n]$. Later, when FindSparse is implemented as a subroutine of CerHAM, the set $F$  consists of the edges that have been queried so far and $E(G_R)$  consists of the edges of $F$ that belong to $G$. FindSparse augments $S$ by recursively adding to it sets that do not (sufficiently) expand in $G_R$ (by adding $Q$ to $S$ we mean replacing $S$ with $S\cup Q$). The variable $j$ and sets $Q_j$  defined in its description should be considered parts of its analysis rather than its description.

\begin{algorithm}[H]
\caption{FindSparse($G,G_R,F,S$)}
\begin{algorithmic}[1]
\\Set $w=1$, $j=1$ and $Q_0=S$.
\While{$w=1$}
\\\hspace{5mm} Set $w=0$.
\\\hspace{5mm} Let $\mathcal{Q}$ be the set of subsets of $V(G)\setminus S$ of size at most $|S|$.
\For{ $Q \in \mathcal{Q}$}
\If{ $w=0$ and $|N_{G_R}(Q)\setminus S|<2|Q|$}
\\ \hspace{15mm} Add $Q$ to $S$ and set $w=1$.
\\ \hspace{15mm} Add to $F$ all the edges incident to $Q$ and query the edges of $G$ that have just been added to $F$. Add any edges of $G$ that have just been revealed  to $G_R$. 
\\ \hspace{15mm} Set $Q_j=S$ and $j=j+1$.
\EndIf
\EndFor
\EndWhile
\\ Return $G_R$, $F$, $S$.
\end{algorithmic}
\end{algorithm}
\begin{observation}\label{obs:runtimeFindSparse}
Let $G$ be such that  $|V(G)|=n$ and $s^*=s^*(G,G_R,F,S)$ be the size of the set $S$ that is returned by  FindSparse($G,G_R,F,S$). If $s^*\leq 0.5n$ then  the running time of  FindSparse is $O((s^*)^3\binom{n}{s^*})$. Indeed if FindSparse outputs a set of size $s^*$ then it updates the set $S$ at most $s^*$ times. After each update, it examines whether there exists a subset of  $[n]\setminus S$ of size at most $|S|\leq s^*$ such that the condition at line 6 is satisfied. There exists at most $\binom{n}{s^*}$ such sets, each examination taking $O((s^*)^2)$ time. Finally, at line 8 it queries at most $ns^*$ edges in total.
\end{observation}

\subsection{Covering non-expanding sets}\label{sub:coveradj}
In this subsection, we introduce the CoverAndAdjust algorithm which is the subroutine that we will use to cover non-expanding sets output either by ReducePaths or by FindSparse.  CoverAndAdjust takes as an input a graph $G'$, a set $S \subseteq V(G')$ and a full path packing $\cP$. It then tries to adjust $\cP$ so that it covers $S$ is found in the interior of some path in $\cP$ (recall Figure \ref{fig2}).

\begin{algorithm}[H]
\caption{$CoverAndAdjust(G',S,\cP)$}
\begin{algorithmic}[1]
\\ Let $G_S'$ be the subgraph of $G'$ whose edge set is exactly the set of edges incident to $S$ in $G'$.
\For{ each component  $C$ of $G_S'$}
\\ \hspace{5mm} Let $Q=Q(C,S)$ be a maximum size subset of $S\cap V(C)$ with the property $|N_{G_S'}(Q)|<2|Q|$.
\\ \hspace{5mm} Set $R=(S\cap V(C))\setminus Q$. 
\\ \hspace{5mm} Let $v$ be a dummy vertex and $F_C$  be the graph with vertex set $V(F_C)=Q\cup N_{G_S'}(Q) \cup \{v\}$ and whose edge set  $E(F_C)$ contains every edge of $G'$ incident to  $Q$ and every edge spanned by $ N_{G_S'}(Q) \cup \{v\}$.
\\ \hspace{5mm} Implement the Inclusion-Exclusion HAM algorithm to find a Hamilton cycle $H$ of $F_C$. If no such a cycle exists output ``$G'$ is not Hamiltonian".
\\ \hspace{5mm} Let $R'= N_{G_S'}(R) \setminus (Q\cup N_{G_S'}(Q))$.
\\ \hspace{5mm} Find a maximum matching $M_1$ in $R\times R'$ and let $R_1'$ be the $M_1$-saturated vertices in $R'$.
\\ \hspace{5mm} Find a maximum matching $M_2$ in $R\times (R'\setminus R_1')$.
\\ \hspace{5mm} Remove from $E(H)$ all the edges not incident to $Q$, then add to it $M_1$ and $M_2$  and let $\cP_C'$ be the set of paths induced by $E(H)$ of length at least $1$.
\\ \hspace{5mm} Split the paths in $\cP$ by removing from them the vertices that are incident to $E(H)$ and add $\cP_C'$ to $\cP$. 
\EndFor
\\Output $\cP$. 
\end{algorithmic}
\end{algorithm} 
 If there exists a subset $W$ of $R$ that satisfies $|N_{G_{S}'}(W)\cap R'| < 2|W|$ then $|N_{G_{S}'}(W\cup Q)|<2|W\cup Q|$ contradicting the maximality of $Q$. Therefore $|N_{G_{S}'}(W)\cap R'| \geq 2|W|$ for every $W\subseteq R$. Hall's Theorem implies that both $M_1$ and $M_2$ saturate $R$. To construct the matchings $M_1,M_2$ at lines $8$ and $9$, we may use any augmenting path algorithm that runs in $O(2^{|R|})$ time. 
 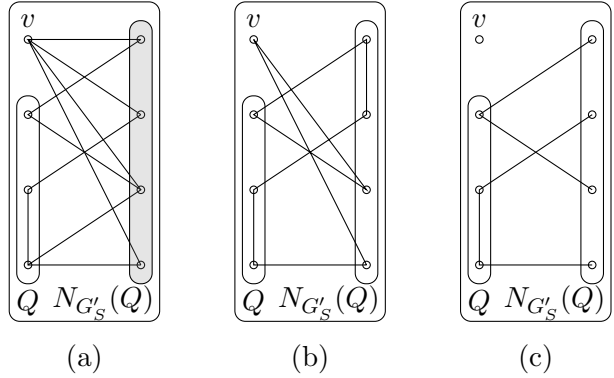
\begin{figure}[htbp]
\begin{center}
\begin{tikzpicture}

\node at (0.75,-1.25) {(a)};

\draw (1.5,0)--(0,0)--(0,1)--(1.5,2);
\draw (0,0)--(1.5,1)--(0,2)--(1.5,3);
\draw (1.5,0)--(0,3)--(1.5,1);
\draw (1.5,2)--(0,3)--(1.5,3);

\draw (0,0) circle [radius=.05];
\draw (0,1) circle [radius=.05];
\draw (0,2) circle [radius=.05];
\draw (0,3) circle [radius=.05];
\node at (0,3.25) {$v$};
\draw (1.5,0) circle [radius=.05];
\draw (1.5,1) circle [radius=.05];
\draw (1.5,2) circle [radius=.05];
\draw (1.5,3) circle [radius=.05];

\draw[rounded corners] (-0.25, -0.75) rectangle (1.75, 3.5) {};
\draw[rounded corners] (-0.15, -0.25) rectangle (0.15, 2.25) {};
\node at (0,-0.5) {$Q$};
\draw[rounded corners, fill=gray, fill opacity=0.2] (1.35, -0.25) rectangle (1.65, 3.25) {};
\node at (1,-0.5) {$N_{G_S'}(Q)$};

\node at (3.75,-1.25) {(b)};

\draw (4.5,0)--(3,0)--(3,1)--(4.5,2)--(4.5,3)--(3,2)--(4.5,1)--(3,3)--(4.5,0);

\draw (3,0) circle [radius=.05];
\draw (3,1) circle [radius=.05];
\draw (3,2) circle [radius=.05];
\draw (3,3) circle [radius=.05];
\node at (3,3.25) {$v$};
\draw (4.5,0) circle [radius=.05];
\draw (4.5,1) circle [radius=.05];
\draw (4.5,2) circle [radius=.05];
\draw (4.5,3) circle [radius=.05];

\draw[rounded corners] (2.75, -0.75) rectangle (4.75, 3.5) {};
\draw[rounded corners] (2.85, -0.25) rectangle (3.15, 2.25) {};
\node at (3,-0.5) {$Q$};
\draw[rounded corners] (4.35, -0.25) rectangle (4.65, 3.25) {};
\node at (4,-0.5) {$N_{G_S'}(Q)$};

\node at (6.75,-1.25) {(c)};

\draw (7.5,0)--(6,0)--(6,1)--(7.5,2);
\draw (7.5,3)--(6,2)--(7.5,1);

\draw (6,0) circle [radius=.05];
\draw (6,1) circle [radius=.05];
\draw (6,2) circle [radius=.05];
\draw (6,3) circle [radius=.05];
\node at (6,3.25) {$v$};
\draw (7.5,0) circle [radius=.05];
\draw (7.5,1) circle [radius=.05];
\draw (7.5,2) circle [radius=.05];
\draw (7.5,3) circle [radius=.05];

\draw[rounded corners] (5.75, -0.75) rectangle (7.75, 3.5) {};
\draw[rounded corners] (5.85, -0.25) rectangle (6.15, 2.25) {};
\node at (6,-0.5) {$Q$};
\draw[rounded corners] (7.35, -0.25) rectangle (7.65, 3.25) {};
\node at (7,-0.5) {$N_{G_S'}(Q)$};

\end{tikzpicture}
\end{center}
\caption{(a)  By placing a complete graph on $N_{G_S'}(Q)$ (the shaded region) one obtains $F_C$. (b) A Hamilton cycle of $F_C$. (c) A set of vertex disjoint paths that covers $Q$. 
}
\label{fig3}
\end{figure}

 The extra vertex $v$ at line $5$ is needed to ensure that the set of edges of $H$ incident to $Q$ does not form a cycle; hence it induces a set of paths in $G'$. To these paths we add the $|R|$ paths induced by $M_1\cup M_2$ to get the paths in $\cP_C'$. The following Lemma is related to line $6$ of CoverAndAdjust. 
\begin{lemma}\label{lem:cover}
Let $G'\subseteq G$ and $S\subset V(G)$ be such that each vertex in $S$ has the same neighborhood in both $G$ and $G'$. If CoverAndAdjust($G',S,\cP$) outputs ``$G'$ is not Hamiltonian" then $G$ is not Hamiltonian.
\end{lemma}

\begin{proof}
Assume that $G$ has a Hamilton cycle $H$ and let $C$ be an arbitrary component of $G_S'$. Then the edges of $H$ incident to the set $Q=Q(C,S)$ (as it is defined in the description of CoverAndAdjust) induce a path packing $\{P_1,P_2,....,P_k\}$ that covers $Q$ such that the endpoints of $P_i$ lie in $N_{G_S'}(Q)$ for $i\in [k]$. As each vertex in $S$ has the same neighborhood in both $G$ and $G'$
these paths are present in $F_C$. Finally as each pair of vertices in $N_{G_S'}(Q)\cup \{v\}$ forms an edge in $F_C$ the paths $P_1,P_2,....,P_{k}$ can be used to form a Hamilton cycle of $F_C$. Thus CoverAndAdjust($G',S,\cP$) does not output ``$G'$ is not Hamiltonian".
\end{proof}

\begin{observation}\label{obs:coverandadjust}
Let $\cP'$ be the union of  the sets $\cP_C'$ over the components $C$ of $G_S'$.
After the execution of CoverAndAdjust$(G',S,\cP)$ the cardinality of $\cP$ may increase by at most $2|S|$. This is because there are at most $|S|$ paths in $\cP'$ covering $S$. These paths cover in total at most $3|S|$ vertices. Removing these at most $3|S|$ vertices from the paths in $\cP$ increases the cardinality of $\cP$ by at most $3|S|$. We then add to $\cP$ the at paths in $\cP'$, causing $|\cP|$ to decrease by the number of edges spanned by $\cP$. As $\cP'$ covers $S$, there are at least $|S|$ such edges.
\end{observation}

A bound on the running time of CoverAndAdjust($G',S,\cP$) can be obtained as follows. Let $s$ be the maximum $|C|$ over subsets $C$ of $S$ with the property that the edges incident to $C$ in $G'$ span a connected graph on $C\cup N_{G'}(C)$ (thus $s$ is equal to the maximum number of vertices of $S$ that are contained in a single component of $G_S'$). 
Then for each component $C'$ of $G_S'$ we may identify $Q$ in $O(s2^s)$ time, apply the Inclusion-Exclusion HAM algorithm at line 6 in $O(s^62^{3s})$ time (here we are using  that $|N_{G_S'}(Q)|<2|Q|\leq 2s $) and run the rest of the lines from $3$ to $11$ in $O(2^s)$ time. 
Hence the running time of  CoverAndAdjust($G',S,\cP$) is   $O(ns^62^{3s})$.

\subsection{Chernoff bounds}
At various places we use the Chernoff bounds, stated below, to bound the probability that the binomial random variable $Bin(n,p)$ deviates from its expectation by a multiplicative factor.
\begin{theorem}
For $n\in \mathbb{N}$, $p=p(n)\in (0,1)$ and $\epsilon>0$,
\begin{equation}\label{eq:chernoffupper}
    \Pr(Bin(n,p)\geq (1+\epsilon)np) \leq \exp\{-\epsilon^2np/(2+\epsilon)\} 
\end{equation}
and
\begin{equation}\label{eq:chernofflower}
    \Pr(Bin(n,p)\leq (1-\epsilon)np) \leq \exp\{-\epsilon^2np/2\}.
\end{equation}
\end{theorem}

\section{A randomized algorithm for certifying Hamiltonicity}\label{sec:RcertifyHam}

RCerHAM starts by generating $G'\sim G(n,0.5)$, letting $F_0$ be the edges of $G'$ that are present in $G$ and setting $G_1=([n],F_0)$.
Then, it runs CerHAM($G,F_0$). The variables $i,j,X_i$ and $Y_j$ defined in the description of CerHAM should be considered parts of its analysis rather than its description.

\begin{algorithm}[H]
\caption{ CerHAM($G,F_0$)}
\begin{algorithmic}[1]
\\ Set $count=0$, $i=j=1$, $S_0=\emptyset$, $F=F_0$ and $\cP=[n]$.
\\ Query the edges in $F$ and set $G_R=([n], E(G)\cap F)$.
\While{ $|S_{i-1}|<0.25n$ and $count< 0.01n^2$}
\\ \hspace{5mm} $End,\{End_u\}_{u\in End}, \cU=$ ReducePaths$(G,\cP,S_{i-1}\cup N_{G_R}(S_{i-1}))$.
\hspace{5mm} \If{ there exists $W_i \in \{End_u:u\in End\}\cup \{End\}$ such that $|W_i|< 0.25n$}
\\\hspace{10mm} Set $S_i'=S_{i-1}\cup W_i$.
\\\hspace{10mm} $G_R$, $F$, $S_i$ = FindSparse($G,G_R,F,S_i'$).
\\ \hspace{10mm} Execute CoverAndAdjust$(G_R,S_i,\cP)$. If it outputs ``$G_R$ is not Hamiltonian" then output ``$G$ is not Hamiltonian"; else set $\cP'=\cP$ and $\cP=$ CoverAndAdjust$(G_R,S_i,\cP')$.
\\ \hspace{10mm} Set   $Y_i=|\cP|-|\cP'|$ and $i=i+1$.
\Else{ }
\If{ there exists $u\in End$ and $w\in End_u$ such that $uw\in F$ and $uw \in G_R$}
\\ \hspace{15mm} Set $E'=\{uw\}$.
\ElsIf{ there exists $u\in End$ and $w\in End_u$ such that $uw\notin F$}
\\ \hspace{15mm} Add $uw$ to $F$ and query whether $uw\in E(G)$. If $uw$ belongs to $E(G)$ then add $uw$ to $G_R$, set $X_j=1$ and $E'=\{uw\}$; else set $X_j=0$ and $E'=\emptyset$.
\\ \hspace{15mm} Set $count=count+1$ and $j=j+1$.
\Else{ go to line 22.}
\EndIf
\If{ $E' \neq \emptyset$ } let $\{uw\}=E'$ and $\cP_{u,w}\in \cU$ be such that none of $u,w$ lies in the interior of a path in $\cP_{u,w}$. Add to the graph spanned by the paths in $\cP_{u,w}$ the edge $uw$. If it spans a Hamilton cycle output ``$G$ is Hamiltonian". Else let $\cP$ be the induced full path-packing. 
\EndIf
\EndIf
\EndWhile
\\Execute the Inclusion-Exclusion HAM algorithm with input the graph $G$.
\end{algorithmic}
\end{algorithm}
Observe that if  CerHAM($G,F_0$) outputs ``$G$ is not Hamiltonian" at line $8$ then this statement is indeed true due to Lemma \ref{lem:cover}. Here we are using that due to line $8$ of FindSparse we have revealed all the edges in $G$ incident to $S_i$, hence the edge sets incident to $S_i$ in $G_R$ and $G$ are identical. The following definition is related to line $16$ of CerHAM.
\begin{notation}
We say that a pair $(G,F_0)$ has the property $\cR$, equivalently $(G,F_0)\in \cR$, if $G$ is a graph, $F_0$ is an edge set and every set $S\subseteq V(G)$ is incident to at most $0.11|S|+n^{7/4}$ edges in $F_0\setminus E(G)$. 
\end{notation}
\begin{lemma}\label{lem:failure}
If $(G,F_0)\in \cR$ then CerHAM($G,F_0$) never executes line 22.
\end{lemma}
\begin{proof}
Let $(G,F_0)\in \cR$. We will prove that throughout the execution of CerHAM($G,F_0$) if the line $13$ is executed then  
the set $\{uw:u\in End, w\in End_u \text{ and } uw \notin F\}$ is not empty, thus the algorithm never executes line 16. For that consider a moment when CerHAM($G,F_0$) proceeds to line $13$.  Observation \ref{lem:posaendpoints} implies that at each iteration of the while-loop at line $3$ the sets $End,\{End\}_{u\in End}$ are disjoint from $S_{i-1}$. In addition observe that $count$ is an upper bound on the number of edges in $F \setminus F_0$  not incident to $S_{i-1}$.  Thus, if the algorithm proceeds to line $13$, then due to line $5$, we have that every set in $\{End_u:u\in End\}\cup \{End\}$ has size at least $0.25n$. In addition due to line $11$, as $F\cap E(G)=F\cap E(G_R)$,  we have that 
\begin{equation*}
|\{(u,w):u\in End, w\in End_u \text{ and }uw \in F\cap E(G)\}|=0.
\end{equation*}
In addition, $(G,F_0)\in \cR$ and therefore
\begin{align*}
    &|\{(u,w):u\in End, w\in End_u \text{ and } uw \in F_0 \setminus E(G)\}|\leq 0.11n|End|+n^{7/4}.
\end{align*}
Hence, as $F_0\subseteq F$, we have 
\begin{align*}
&|\{(u,w):u\in End, w\in End_u \text{ and } uw \notin F\}|
\\&= |\{(u,w):u\in End, w\in End_u \}|
\\&- |\{(u,w):u\in End, w\in End_u \text{ and } uw \in F \setminus F_0\}|
\\&- |\{(u,w):u\in End, w\in End_u \text{ and } uw \in F_0 \setminus E(G)\}| 
\\&-|\{(u,w):u\in End, w\in End_u \text{ and }uw \in F_0\cap E(G)\}| 
\\& \geq 0.25n|End|-2count - 0.11n|End|-n^{7/4}-0
\\& \geq 0.14n|End| -0.021n^2\geq 0.14\cdot 0.25 n^2 -0.021n^2\geq 1.
\end{align*}
\end{proof}

The following lemma upper bounds the number of edges in $G$ we need to identify at line 15. We later use Lemma \ref{lem:auxrandupper} to upper bound the probability that $count \geq 0.01n^2$ at the end of the main while-loop of  CerHAM($G,F_0$). Given a set of edges $F_0$  and a graph $G$ on $[n]$ define the event $\cE_{exp}'(G,F_0)$ as follows. $$\cE_{exp}'(G,F_0)=\set{\exists S\subset [n]: |S|\in [0.02n, 0.27n]\text{ and }|N_{G \cap G_{F_0} }(S)|<2|S|},$$
where $G_{F_0}=([n],F_0)$.
\begin{lemma}\label{lem:auxrandupper}
Let $G\sim G(n,p)$ and $F_0$ be an edge set on $[n]$ with the property $(G,F_0)\in \cR$. In addition let $G_1=([n],F_0\cap E(G))$. Let $t$ be the largest $i$ such that the set $S_i$ is defined by CerHAM($G,F_0$) and $S_0,S_1,....,S_t$ be the corresponding sets. Then, $|N_{G_1}(S_t)|< 2|S_t|.$ In addition, if the event $\cE_{exp}'(G,F_0)$ does not occur then 
$|S_t|\leq 0.02n$. Furthermore, 
\begin{equation}\label{eq:trand}
t\leq \log_2\bfrac{|S_t|}{|S_1|}+1.    
\end{equation}
Finally, if CerHAM($G,F_0$) exits the while-loop and $|S_t|<0.25 n$  then, 
\begin{equation}\label{eq:probrand}
    \sum_{i=1}^{0.01n^2}X_i < 4|S_t|+n.
\end{equation}
\end{lemma}

\begin{proof}
Let $\emptyset=Q_0,Q_1,....,Q_r=S_t$ be the sequence of sets generated by  CerHAM($G,F_0$), i.e. for $1\leq i\leq r$ either (i) $Q_{i-1}=S_j$, $Q_i\setminus Q_{i-1}=W_{j+1}$ and $Q_i=S_{j+1}'$ for some $1\leq j+1\leq t$ or (ii) $Q=Q_i\setminus Q_{i-1}$ is a subset of $[n]\setminus Q_{i-1}$ of size at most $|Q_{i-1}|$ such that $|N_{G_R}(Q)\setminus Q_{i-1}|<2|Q|$ for some graph $G_R$ with $G_1\subseteq G_R\subseteq G$ (i.e the set $Q$ satisfied line $6$ of FindSparse at some point during the execution of  CerHAM($G,F_0$)). Lemma \ref{lem:posa} implies that in the first case, and therefore in both cases, $|N_{G_1}(Q_i\setminus Q_{i-1})\setminus Q_{i-1}|<2|Q_i\setminus Q_{i-1}|$ and $|Q_i\setminus Q_{i-1}|\leq \max\{|Q_{i-1}|,0.25n-1\}$.

We first show by induction that $|N_{G_1}(Q_i)|<2|Q_i|$ for $1\leq i\leq r$. Taking $i=r$ yields that $|N_{G_1}(S_t)|< 2|S_t|$. The base case holds as $Q_1\setminus Q_0=Q_1=W_1$. Assume that $|N_{G_1}(Q_i)|<2|Q_i|$ for some $1\leq i< r$. Then,
\begin{align*}
    |N_{G_1}(Q_{i+1})|&\leq |N_{G_1}(Q_{i+1}\setminus Q_{i})\setminus Q_i|+ |N_{G_1}(Q_{i}|
<2|Q_{i+1}\setminus Q_{i}|+2|Q_i|\leq 2|Q_{i+1}|,
\end{align*}
completing the induction.

Now assume that the event $\cE_{exp}'(G,F_0)$ does not occur. We will show that $|Q_i|\leq 0.02n$ for $0\leq i \leq r$ by induction. Taking $i=r$ yields that $|S_t|< 0.02n$. The base case holds as $Q_0=\emptyset$. Assume that $Q_i<0.02n$ for some $0\leq i< r$. Then,
\begin{align*}
    |Q_{i+1}|&= |Q_{i+1}\setminus Q_{i}|+ |Q_{i}|
 \leq \max\{|Q_{i}|,0.25n-1\} +|Q_i|
< 0.25n-1+0.02n< 0.27n.
\end{align*}
As $|N_{G_1}(Q_{i+1})|< 2|Q_{i+1}|$ and the event $\cE_{exp}'(G,F_0)$ does not occur we have that $|Q_{i+1}|<0.02n$.

Now we will show that $2|S_i|<|S_{i+1}|$ for $0\leq i<t$.  
For that observe that once FindSparse($G,G_R,F,$ $S_i'$) terminates there does not exists a subset $W$ of $[n]\setminus S_i$ of size at most $|S_i|$ with the property $|N_{G_R}(W)\setminus S_i|<2|W|$. On the other hand Lemma \ref{lem:posa} implies that $|N_{G_R}(W_{i+1})\setminus S_i|<2|W_{i+1}|$. Therefore, $|W_{i+1}|>|S_i|$ and
$$|S_{i+1}|\geq |S_{i+1}'|=|S_i|+|W_{i+1}|>|S_i|+|S_i|=2|S_i|,$$
for $1\leq i\leq t$.
$2|S_i|<|S_{i+1}|$ for $0\leq i<t$ implies that 
$|S_t|\geq 2^{t-1}|S_1|$ and therefore \eqref{eq:trand} holds.

Finally for \eqref{eq:probrand} observe that if  CerHAM($G,F_0$) exits the while-loop and $|S|<0.25n$ then  line 15 has been executed $0.01n^2$ times. Say $|\cP|=0$ if a Hamilton cycle has been constructed. The $i^{th}$ time line 8 is executed the number of paths in $\cP$ is increased by $|Y_i|$.  Observation \ref{obs:coverandadjust} implies that $|Y_i|\leq 2|S_i|$. On the other hand  $X_j=1$ implies a decrease in $|\cP|$ after the $j^{th}$ time line 15 is executed.  
If  CerHAM($G,F_0$) exits the while-loop, then $|\cP|>0$ throughout the algorithm and as initially $|\cP|=n$ we have,
\begin{align*}
    0&< n+\sum_{i=1}^tY_i- \sum_{j=1}^{0.01n^2}X_i
    \leq  n +2 \sum_{i=1}^t |S_i| - \sum_{i=1}^{0.01n^2}X_i
    \\&\leq n+2 \sum_{i=1}^t2^{-(t-i)}|S_t| - \sum_{i=1}^{0.01n^2}X_i \leq n+4|S_t| - \sum_{i=1}^{0.01n^2}X_i.
\end{align*}
At the second line of the calculations above we used that $2|S_i|<|S_{i+1}|$ for $0\leq i<t$.
\end{proof}

\subsection{Analysis of CerHAM in the randomised setting - the dense regime}
For the analysis of CerHAM($G,F_0$) in addition to the event $\cE_{exp}'(G,F_0)$ we consider the events $\cE_{count}(G,F_0)$ and $\cA_i'(G,F_0)$, $i\leq 0.02n$.
$\cE_{count}(G,F_0)$ is the event that $count$ reaches  the value of $0.01n^2$ at the execution of CerHAM($G,F_0$)
and  $\cA_i'(G,F_0)$ is the event that $\exists S\subset [n]$ such that $|S|=i$ and 
 $|N_{ F\cap G_{F_0} }(S)|<2|S|$.
 
For a graph $G$ on $[n]$ and an edge set $F_0$ on $[n]$ we let $T(G,F_0)$ be the running time of CerHAM$(G,F_0)$. In addition, we let $T_A(G,F_0)$ be the running time of CerHAM in the event it does not execute line $6$ (thus it does not execute the subroutines FindSparse and CoverAndAdjust at all) and $T_B=T-T_A(G,F_0)$.
\begin{lemma}\label{lem:upperT}
Let $G$ be a graph on $[n]$ and $F_0$ be an edge set on $[n]$ with the property $(G,F_0)\in \cR$.  Then $$T_A(G,F_0)=O(n^7)$$ and
\begin{align}\label{eq:upperTB}
\mathbb{E}(T_B(G,F_0))=  O\bigg(n^72^n&\Pr(\cE_{exp}'(G,F_0)\lor  \cE_{count}(G,F_0)) \nonumber\\&+\sum_{j=0}^{0.02n}\bigg[ j^4 \binom{n}{j} + j^7 2^{3j}n\bigg]
\Pr(\cA_j'(F_0))\bigg).
\end{align}
\end{lemma}
\begin{proof}
Let $(G,F_0)\in \cR$ and $G_1=([n],F_0\cap E(G))$. Write $\cE_{exp}',\cE_{count}$ and $\cA_j'$ for the events $\cE_{exp}'(G,F_0),\cE_{count}(G,F_0)$ and $\cA_j'(G,F_0)$ respectively. 
Lemma \ref{lem:auxrandupper} implies that $t\leq \log_2 n$ and therefore CerHAM($G, F_0$) may execute line $4$ at most $\log_2 n+ 0.01n^2$ times. Each execution of line $4$ runs in $O(n^3)$ time; hence line $4$ takes in total $O(n^5)$ time.
This is also equal to $T_A(G,F_0)$.

In the event $\cE_{exp}'\lor \cE_{count}$ the algorithm may exit the while loop. Before exiting the while-loop, it may reach line 6 at most $\log_2n < n$ times, each time spending at most $O(n^62^n)$ time at lines $6$ to $9$. After exiting the while-loop, it executes the Inclusion-Exclusion HAM algorithm whose complexity is $O(n^62^n)$. Hence  in the event  $\cE_{exp}'\lor \cE_{count}$ CerHAM($G, F_0$) runs in $O(n^72^n)$ time. This is also the worst-case running time of CerHAM($\cdot, \cdot$).

On the other hand, if none of $\cE_{count}$, $\cE_{exp}'$ occurs, by Lemma \ref{lem:auxrandupper} we have that  $|S_t|\leq 0.02n$. Hence we may partition the event $\neg \cE_{exp}' \land \neg \cE_{count}$ into the events $\{\cF_j\}_{j=0}^{0.02n}$ where $\cF_j=\{|S_t|=j\} \land \neg \cE_{exp}\land \neg\cE_{count}$. In the event $\cF_j$ CerHAM($G, F_0$) executes lines $6$ to $9$ at most $t\leq |S_t|=j$ times. Each of the at most $j$ executions of FindSparse runs in $O(j^3\binom{n}{j})$ time (see Observation \ref{obs:runtimeFindSparse}).  Thereafter, each of the executions of CoverAndAdjust at line $8$ runs in $O(n j^6 2^{3j})$ time (see the last paragraph of Section 2).
Thus lines $6$ to $9$ are executed in $O(j^4\binom{n}{j}+j^7 2^{3j}n)$ time.

Equation \eqref{eq:upperTB} follows from the observation that if the event $\cF_j$ occurs then the event $\cA_j' \land \neg \cE_{exp} \land  \neg\cE_{count}$ also does and therefore $\Pr(\cF_j)\leq \Pr(\cA_j' \land \neg \cE_{exp}' \land \neg \cE_{count})\leq \Pr(\cA_j')$.
\end{proof}

\begin{lemma}\label{lem:uppercalcrand}
Let $G\sim G(n,p)$, $G'\sim G(n,0.5)$ and $F_0=E(G)\cap E(G')$. Then $T_B(G, F_0)=O(1)$ for $ p\geq \frac{100\log n}{n}$.
\end{lemma}

\begin{proof}

$F_0\subseteq E(G)$ implies that $(G,F_0)\in \cR$. Let $G_1=([n],F_0)$. 
Write $\cE_{exp}',\cE_{count}$ and $\cA_j'$ for the events $\cE_{exp}'(G,F_0),\cE_{count}(G,F_0)$ and $\cA_j'(G,F_0)$ respectively. 
In the event $\cE_{exp}'$  there exists a set $S\subset [n]$  with size in $[0.02n, 0.27n]$  such that $|N_{ G_1}(S)|<2|S|$. Therefore $|S|\geq 0.02n$, $|[n]\setminus (S\cup N_{G_1}(S))|\geq (1-0.27\cdot3)n=0.19n$ and no edge from $S$ to $[n]\setminus (S\cup N_{G_1}(S))$ belongs to $G_1$. Thus,
\begin{align*}
\Pr(\cE_{exp}')&\leq 2^n \cdot 2^n \cdot (1-0.5p)^{0.02n\cdot 0.19 n}\leq 4^n e^{-10^{-3} pn^2}=o(n^{-7}2^{-n}).
\end{align*}
In the event $\cE_{count}\setminus \cE_{exp}'$ \eqref{eq:probrand} implies that $\sum_{i=1}^{0.01n^2}X_i<n+4\cdot 0.02n\leq 1.1n$. For $1\leq i\leq 0.01n^2$, $X_i=1$ only if the corresponding edge belongs to $G$ but not to $G'$, hence with probability $0.5p$ independently of $X_1,X_2,...,X_{i-1}$. Thus, as $p\geq \frac{100\log n}{n}$, the Chernoff bound gives,
\begin{align*}
    \Pr(\cE_{count}\setminus \cE_{exp}')& \leq \Pr\bigg(Bin\bigg(0.01n^2,\frac{50\log n }{n}\bigg) \leq 1.2n\bigg)=o(n^{-7}2^{-n}).
\end{align*}
Finally, for $j\leq 0.02n$, in the event $\cA_j'$ one may identify sets $S$ and $W$ of size $j$ and $2j$ respectively such that $N_{G_1}(S)\subset W$. In addition note that for $j\leq 0.02n$ we have that $ j^{3}2^{3j}n \leq 8 \binom{n}{j}$. Therefore,
\begin{align*}
    \bigg[j^4 \binom{n}{j} +j^72^{3j}n\bigg]
\Pr(\cA_j') &\leq 10j^4\binom{n}{j} \binom{n}{j} \nonumber \binom{n}{2j}(1-0.5p)^{j\cdot (n-3j)} 
\\&\leq 10j^4\cdot n^{4j}e^{-0.5pj(1-0.06)n} \leq 10j^4\cdot n^{4j}e^{-45j\log n} 
=O(1).
\end{align*}
At the last inequality we used that $p\geq \frac{100\log n}{n}$. \eqref{eq:upperTB} and the above calculations  imply that $T_B(G, F_0)=O(1)$ for $ p\geq \frac{100\log n}{n}$.
\end{proof}
\textbf{Proof of Theorem \ref{thm:Rmain}:} At the execution of RCerHAM($G$) we have that $F_0\subset E(G)$ hence $(G,F_0)\in \mathcal{R}$. $G(n,0.5)$ may be generated in $O(n^7)$ time. Thereafter, the decomposition of the expected running time of RCerHAM($G$) follows from lemmas \ref{lem:upperT} and \ref{lem:uppercalcrand}.
\qed
\section{Derandomizing RCerHAM}\label{sec:derandom}

The only place where RCerHAM uses any sort of randomness is in generating the set $F_0$ before  executing  CerHAM. There, it generates $G'\sim G(n,0.5)$ and lets $F_0=E(G)\cap E(G')$. To derandomize RCerHAM we substitute $E(G')$ with a deterministic set of edges $F_0$ that has the property $(G,F_0)\in \cR$.

For that  we let $F_0=E(H_n)$ where $H_n$ is a deterministic graph on $[n]$. It is constructed by taking a $d$-regular pseudorandom graph $H_n'$ on at least $n$ and at most $4n$ vertices, then taking the subgraph of $H_n'$ induced by a subset of $V(H_n')$ of size $n$ and finally adding a number of edges incident to vertices of small degree. The construction of $\{H_n\}_{n\geq 1}$ as well as the proof of the lemma that follows are given in Appendix \ref{app:pseudorandom}.

\begin{lemma}\label{lem:pseudo}
For sufficiently large $n$ there exists a graph $H_n$ on $[n]$ that can be constructed in $O(n^7)$ time and satisfies the following.
\begin{itemize}
    \item[(a)] For every pair of disjoint sets $U, W\subseteq [n]$ the number of edges spanned by $U\times W$ is at least $\frac{0.1|U||W|}{n} -n^{7/4}$ and at most $\frac{0.101|U||W|}{n} +n^{7/4}$
    \item[(b)] $H_n$ has minimum degree $0.1n$.
    \item[(c)] At most $n^{2/3}$ vertices in $[n]$ have degree larger than $0.101n$ in $H_n$. 
\end{itemize}
\end{lemma}

\subsection{A deterministic algorithm for certifying Hamiltonicity}\label{sec:certifyHam}
DCerHAM starts by setting $G'=H_n$ and $F_0=E(H_n)$. Then it implements CerHAM($G,F_0$). The analysis of DCerHAM is identical to the analysis of RCerHAM modulo the calculations done for bounding the probabilities of events that depend on $F_0$ and $G$.

\subsection{Analysis of CerHAM in the deterministic setting - the dense regime}

\begin{lemma}\label{lem:uppercalc}
Let  $G\sim G(n,p)$, $G'=H_n$ and $F_0=E(H_n)$. Then, $T_B(G, F_0)=O(1)$ for $ p\geq \frac{100\log n}{n}$.
\end{lemma}

\begin{proof}
Lemma \ref{lem:pseudo} implies that $(G,F_0)\in \cR$. Thus, it suffices to verify that the expression given in $\eqref{eq:upperTB}$ is $O(1)$.  Let $G_1=([n],F_0\cap E(G))$. Write $\cE_{exp}',\cE_{count}$ and $\cA_j'$ for the events $\cE_{exp}'(G,F_0),\cE_{count}(G,F_0)$ and $\cA_j'(G,F_0)$ respectively. 

In the event $\cE_{exp}'$  there exists a set $S\subset [n]$  with size in $[0.02n, 0.27n]$  such that $|N_{ G_1}(S)|<2|S|$. Therefore $|S|\geq 0.02n$, $|[n]\setminus (S\cup N_{G_1}(S))|\geq (1-0.27\cdot3)n=0.19n$ and no edge from $S$ to $[n]\setminus (S\cup N_{G_1}(S))$ belongs to $G_1$. Lemma \ref{lem:pseudo} implies that there are at least $0.1\cdot 0.02n \cdot 0.19n -o(n^2) \geq 10^{-4}n$ edges from $S$ to $[n]\setminus (S\cup N_{G_1}(S))$ in $E(H_n)$. Each of these edges belongs to $G$, hence to $G_1$, independently with probability $p$. Thus,
\begin{align*}
\Pr(\cE_{exp}')&\leq 2^n \cdot 2^n \cdot (1-p)^{10^{-4}n^2} \leq 4^n e^{-10^{-4} pn^2}=o(n^{-7}2^{-n}).
\end{align*}
In the event $\cE_{count}\setminus \cE_{exp}'$ \eqref{eq:probrand} implies that $\sum_{i=1}^{0.01n^2}X_i<n+8\cdot 0.02n\leq 1.2n$. For $1\leq i\leq 0.01n^2$, $X_i=1$ only if the corresponding edge $uv$ belongs to $G$ but not to $G'$ ($e\notin F$ hence $e\notin E(H_n) \subseteq F$ due to line 11 of CerHAM), hence with probability $p$ independently of $X_1,X_2,...,X_{i-1}$. Thus, using the Chernoff bound, we have,
\begin{align*}
    \Pr(\cE_{count}\setminus \cE_{exp}')& \leq \Pr(Bin(0.01n^2,p)\leq 1.2n)=o(n^{-7}2^{-n}).
\end{align*}
Finally, for $j\leq 0.02n$, in the event $\cA_j'$ one may identify sets $S$ and $W$ of size $j$ and $3j$ respectively such that $S\cup N_{G_1}(S)\subset W$. $H_n$ has minimum degree $0.1n$ and therefore if $j\leq 10^{-5}n$ then there exists at least $0.09jn$ edges from $S$ to $[n]\setminus W$ in $E(H_n)$. On the other hand, if $j\geq 10^{-5}n$,
Lemma \ref{lem:pseudo} implies that there exist  at least $0.1\cdot j \cdot (n-3j)-o(n^2)\geq 0.09jn$ edges from $S$ to $[n]\setminus W$. As none of them belongs to $G$ in the event $\cA_j'$, we have,
\begin{align*}
   \bigg[j^4 \binom{n}{j} +j^72^{3j}n\bigg]
\Pr(\cA_j') \nonumber
 &\leq 10j^4\binom{n}{j} \binom{n}{j} \nonumber \binom{n}{2j}(1-p)^{0.09jn} 
\\&\leq 10j^4\cdot n^{4j}e^{-0.09pn j} \leq 10j^4\cdot n^{4j}e^{-9j\log n} 
=O(1).
\end{align*}
At the last inequality we used that $p\geq \frac{100\log n}{n}$.

\eqref{eq:upperTB} and the above calculations  imply that $T_B(G, F_0)=O(1)$ for $ p\geq \frac{100\log n}{n}$.
\end{proof}

\textbf{Proof of Theorem \ref{thm:main} for $p\geq \frac{100\log n}{n}$:} At the execution of DCerHAM($G$), Lemma \ref{lem:pseudo} implies that $(G,F_0)\in \mathcal{R}$. DCerHAM($G$) generates $H_n$ in $O(n^7)$ time. Then, the decomposition of the expected running time of DCerHAM($G$) follows from lemmas \ref{lem:upperT} and \ref{lem:uppercalc}.
\qed

\section{The DCerHAM algorithm - bypassing the Hamiltonicity threshold}

For a graph $G$ we let $V_6(G)'$ be the $6$-core of $G$ i.e., the maximal subset $S$ of $V(G)$ with the property that every vertex in $S$ has at least $6$ neighbors in $S$. We also let $V_6(G)=V(G)\setminus V_6(G)'$. It is well known that $V_6(G)'$ can be identified by a peeling procedure, where one recursively removes from $G$ vertices whose current degree is at most $5$. $V_6(G)'$ is the set of vertices remaining at the end of this peeling procedure. We extend the description of the DCerHAM algorithm first to the range $(500\log\log n)/{n}\leq p\leq (100\log n)/{n}$, done at this section, and then to the range $5000/n\leq p\leq ( 500\log \log n)/n$, done at the next section. At both ranges we consider $V_6(G_1)$ where $G_1=([n],E(H_n)\cap E(G))$. 

The basic idea that allows us to extend the analysis of DCerHAM is the following. Initially, add to $S_0$ all the vertices in $V_6(G_1)$. Thereafter any vertex that is added to $S_0$ belongs to some set $W$ with the property $|N_{G_1}(W)\setminus V_6(G)|< 2|W|$. Note that that every vertex in $[n]\setminus V_6(G_1)$ has degree $6$. Thus, if at the end the set $S_t\setminus V_6(G_1)$ has size $i$ then we have identified a set whose neighborhood in $G_1[V(G)\setminus V_6(G_1)]$ has size $j<2i$ and which spans, together with its neighborhood in $G_1$, 
at least $(6i+j)/2 \geq 2i+ (2i+j)/2>2i+j$ edges in $G_1$. This will be the new moderately unlikely event that we will consider.

\subsection{Analysis of DCerHAM  - the middle range}\label{sub:middle:algo}
Given a graph $G$ and a set of edges $F_0$ such that $(G,F_0)\in \cR$ the DCerHAM algorithm with input $G,F_0$ executes the following steps. First it lets $G_1=([n],F_0\cap E(G))$ and $K=V_{6}(G_1)$. If $|K|< \frac{n}{10\log_2 n}$ then it lets $S_0=K$ and $F=F_0$.  Else it sets $G_R,F,S'=$FindSparse$(G,G_1,E(G_1),K)$ and  $S_0=K\cup S'$.  We consider the case distinction $|K|<\frac{n}{10\log_2 n}$ and $|K|\geq \frac{n}{10\log_2 n}$ for technical reasons that will become apparent at the proof of Lemma \ref{lem:coundmiddlerandom}.  In both cases, it implements CoverAndAdjust and finds an FPP $\cP_0$ that covers $S_0$ or outputs that $G$ is not Hamiltonian.
Then it continues by executing the CerHAM algorithm. 
Now the CerHAM algorithm takes as input the graph $G$ and the sets $F,S_0$ and $\cP_0$. CerHAM now runs with the following two modifications. First, at line $1$ it sets $S=S_0$ and $\cP=\cP_0$ (instead of $S=\cP=\emptyset$). Second, henceforward every time it executes the FindSparse algorithm, at line $4$ of FindSparse, it lets $\cQ$ be the set of subsets of $V(G)\setminus S_0$ of size at most $|S\setminus S_0|$ (instead of size at most $|S|$).

For a pair of graphs $G_1,G_2$ on the same vertex set $V$ we let $G_1\cap G_2=(V,E(G_1)\cap E(G_2))$. Given a set of edges $F_0$ and a graph $G$ on $[n]$, with $G_{F_0}=([n],F_0)$, define the events:
\begin{align*}
    \cE_{s}(G,F_0)=\{ |V_{6}(G\cap G_{F_0})|\geq s) \},
\end{align*}
\begin{align*}
    \cE_{exp}(G,F_0)=\{&\exists S\subset [n]\setminus V_{6}(G\cap G_{F_0}) : 0.02n\leq S\leq 0.27n 
    \\&\text{ and } |N_{G\cap G_{F_0}}(S)\setminus  V_{6}(G\cap G_{F_0})|< 2|S|\},
\end{align*}
\begin{align*}
    \mathcal{A}_i(G,F_0):=\{ 
  &\exists W \subseteq [n]\setminus V_{6}(G \cap G_{F_0})\text{ such that $i\leq |W| \leq 0.02n$ and }  
  \\&  |N_{G \cap G_{F_0}}(W)\setminus V_{6}(G \cap G_{F_0})|<2|W|\},
\end{align*}
and 
\begin{align*}
    \mathcal{B}_i(G,F_0):= \{&\exists W \subseteq [n]\text{ such that }|N_{G\cap G_{F_0}}(W)|\leq 7|W| 
    \\&\text{and $W\cup N_G(W)$ is connected in $G$}\}.
\end{align*}
The above events are used in the analysis of DCerHAM in the following manner. $K=V_{6}(G \cap G_{F_0})$. Thus 
$\cE_{s}(G,F_0)$ states that  $|K|\geq s$, while $\cE_{exp}(G,F_0)$ implies that we have identified a medium size non-expanding set in $G[V(G)\setminus K]$, $S_i\setminus K$ is such a set. Once again, the event $\cA_i$ corresponds to the unlikely event that we have identified a non-expanding set of size $i$, now in $G[V(G)\setminus K]$. Finally, the events $\mathcal{B}_i(G,F_0)$, $i\leq 0.03n$ are used to upper bound the expected running time of CoverAndAdjust, say $T$. We will bound $T$ by a function that depends on the largest component spanned by $S_t\cup N_{G_R}(S_t)$, say $C$, as opposed to the size of $S_t$. It will turn out that if $|C|=i$ then the event $\cB_i(G,F_0)$ occurs.

We upper bound the probability of the above events occurring at the following Lemma. Its proof is located at Appendix \ref{app:sec:6core}
\begin{lemma}\label{lem:color}
Let $\frac{5000}{n}\leq p \leq \frac{500\log\log n}{n}$, $G\sim G(n,p)$, $G'=H_n$ and $F_0=E(G)\cap E(G')$. Then  for $1\leq i\leq 0.02n$, $\frac{\log n}{100} \leq j \leq 0.03n$ and $\frac{n}{\log^2 n}\leq s\leq 0.01n$ the following hold.
\begin{equation}\label{eq:4core}
 \Pr(\cE_{s}(G,F_0))\leq  \bigg( \bfrac{en}{s} \bfrac{0.1enp}{5}^5 e^{-0.09np}\bigg)^s, \hspace{10mm} \Pr(\cE_{exp}(G,F_0)\lor \cE_{0.01n})\leq 2^{-1.1n},
\end{equation}
\begin{equation}\label{eq:big_s}
i^4 \binom{n}{i}\Pr(\cA_i(G,F_0))=O(1) \hspace{5mm} \text{ and } \hspace{5mm} j^72^{3j}\Pr( \cB_j(G,F_0)) \leq e^{-100j}.
\end{equation}
\end{lemma}

In place of Lemma \ref{lem:auxrandupper}, we have the following one.   
\begin{lemma}\label{lem:auxlower}
Let  $G$ be a graph on $[n]$  and $F_0$ be an edge set on $[n]$ with the property $(G,F_0)\in \cR$. In addition let $G_1=([n],F_0\cap E(G))$. Let $t$ be the larger $i$ such that the set $S_i$ is defined by CerHAM($G,F_0$) and $S_0,S_1,....,S_t$ be the corresponding sets. Then, $|N_{G_1}(S_t\setminus K)\setminus K|< 2|S_t\setminus K|.$ In addition, if the events $\cE_{exp}(G,F_0), \cE_{0.01n}(G,F_0)$  do not occur then $|S_t\setminus K|\leq 0.02n$. Furthermore,
\begin{equation}\label{eq:tlower}
t\leq \log_2\bfrac{|S_t|-|S_0|}{|S_1|-|S_0|}+1.    
\end{equation}
Finally, if  CerHAM exits the while-loop and $|S_t|<0.25 n$  then, 
\begin{equation}\label{eq:Xilower}
    \sum_{i=1}^{0.01n^2}X_i < 2t|S_0|+ 4|S_t|+n.
\end{equation}
\end{lemma}

\begin{proof}
The proof of Lemma \ref{lem:auxlower} is almost identical to the proof of Lemma \ref{lem:auxrandupper}. Here we only present only parts of the proof of \eqref{eq:Xilower}. 

Write $\cE_{exp}$ and $\cE_{0.01n}$ for the events $\cE_{exp}(G,F_0)$ and $\cE_{0.01n}(G,F_0)$ respectively.
As in  the proof of \eqref{eq:probrand} one has $2(|S_i|-|S_0|)=2|S_i\setminus S_0|< |S_{i+1}\setminus S_0|=|S_{i+1}|-|S_0|$ for $0\leq i<t$ and  
\begin{align*}
    0&< n+\sum_{i=1}^tY_i- \sum_{j=1}^{0.01n^2}X_i
    \leq  n +2 \sum_{i=1}^t |S_i| - \sum_{i=1}^{0.01n^2}X_i
    \leq  n +2t|S_0|+2 \sum_{i=1}^t (|S_i|-|S_0|) - \sum_{i=1}^{0.01n^2}X_i
    \\&\leq n+2t|S_0|+2 \sum_{i=1}^t2^{-i+t}(|S_t|-|S_0|) - \sum_{i=1}^{0.01n^2}X_i \leq n+2t|S_0|+4|S_t| - \sum_{i=1}^{0.01n^2}X_i.
\end{align*}
\end{proof}

\begin{lemma}\label{lem:coundmiddlerandom}
Let $G$ be a graph on $[n]$ and $F_0$ be an edge set on $[n]$ with the property $(G,F_0)\in \cR$. If the event $\neg \cE_{exp}(G,F_0)\land \neg \cE_{0.01n}(G,F_0)$ occurs then $\sum_{i=1}^{0.01n^2}X_i\leq 2n$.
\end{lemma}

\begin{proof}
Write  $\cE_{exp}$ and $\cE_{0.01n}$  for the events  $ \cE_{exp}(G,F_0)$ and $\cE_{0.01n}(G,F_0)$ respectively. Let $S_0,S_1,...,S_t$ be as in the previous lemma, $G_1=([n], E(G)\cap F_0)$ and $K=V_{6}(G_1)$. 

First assume that the event $\neg \cE_{exp}\land \neg \cE_{0.01n}  \land \set{|K| < \frac{n}{10\log_2 n}}$ occurs. Then $K=S_0$. In addition, by Lemma \ref{lem:auxlower}, we have that in the event $\neg \cE_{exp}\land \neg \cE_{0.01n}$  the set $S_t$ has size at most $0.02n $ and $t\leq \log_2 |S_t|+1$. Hence \eqref{eq:Xilower} gives,
$$\sum_{i=1}^{0.01n^2} X_i < 2t|S_0|+4|S_t|+n \leq 2(\log_2n +1)\cdot \frac{n}{10\log_2 n}+0.08n +n\leq 2n.$$
On the other hand, if the event $\neg \cE_{exp}\land \neg \cE_{0.01n}  \land \set{|K| \geq \frac{n}{10\log_2 n}}$ occurs then the algorithm FindSparse($G,G_1,E(G_1),K)$ is executed. Upon termination of FindSparse($G,G_1,E(G_1),K)$, $S=S_0$ and there does not exists a subset $Q$ of $[n]\setminus S_0$ of size at most $|S_0|$ with the property that $|N_{G_1}(Q)\setminus S_0| <2|Q\setminus S_0|$. Since $W_1$ has this property we have that $|W_1|\geq |S_0|$ and $|S_1|-|S_0|= |S_1\setminus S_0|\geq |W_1|\geq |S_0|$. \eqref{eq:tlower}  implies that 
$$t\leq \log_2 \bfrac{|S_t\setminus S_0|}{|S_1\setminus S_0|} \leq \log_2\bfrac{|S_t|}{|S_0|} \leq \log_2\bfrac{0.02n}{|S_0|}.$$ Recall that $|S_0|\leq |S_t|\leq 0.02n$ in the event $\neg \cE_{exp}\land \neg \cE_{0.01n}$, thus combining the above inequality with \eqref{eq:Xilower} we get, $$\sum_{i=1}^{0.01n^2}X_i<2n\cdot \max\set{x\log_2\bfrac{0.02}{x} :{0\leq x\leq 0.02}}+4\cdot 0.02n+n<2n.$$
\end{proof}

For a graph $G$ on $[n]$ and an edge set $F_0$, let $T(G,F_0)$ be the running time of DCerHAM$(G,F_0)$. In addition  let $T_B(G,F_0)$ be the portion of the running time of DCerHAM spend on executing FindSparse, executing  CoverAndAdjust$(G_R,S_i,\cP)$ for pairs $(G_R, S_i)$ such that the maximum subset $C$ of $S_i$ with the property that $C\cup N_{G_R}(C)$ is connected in $G_R$ has size at least $\log n/10$, and executing line $22$ of CerHAM. Finally, we let $T_A(G,F_0)=T(G,F_0)-T_B(G,F_0)$. Recall, given a set of edges $F_0$ and a graph $G$ on $[n]$ the event $\cE_{count}(G,F_0)$ is defined by,
$$\cE_{count}(G,F_0)=\{ count \text{ reaches  the value of } 0.01n^2\text{ at the execution of CerHAM($G,F_0$)}\}.$$
In place of Lemma \ref{lem:upperT} we have the following one.

\begin{lemma}\label{lem:middleT} 
Let $G$ be a graph on $[n]$ and $F_0$ be an edge set on $[n]$ with the property $(G,F_0)\in \cR$.  Then $$T_A(G,F_0)=O(n^7)$$ and
\begin{align}\label{eq:middleTB}
\mathbb{E}(T_B(G,F_0))=  
   O\bigg(   &n^72^n\Pr(\cE_{exp}(G,F_0)\lor \cE_{count}(G,F_0) \lor \cE_{0.01n}(G,F_0))
 \\& +\sum_{i=\frac{n}{10\log_2 n}}^{0.01n}i^4\binom{n}{i} \Pr\big(|K|=i)+\sum_{i=0}^{0.02n }  i^4 \binom{n}{i}\Pr(\cA_i(G,F_0) ) \nonumber
   \\&+\sum_{i=\frac{\log n}{10}}^{0.02n }n^2i^62^{3i} \Pr(\cA_i(G,F_0))  +\sum_{i=\frac{\log n}{10}}^{0.03n }n^2i^62^{3i} \Pr(\cB_i(G,F_0)) \bigg) \nonumber. 
\end{align}
\end{lemma}

\begin{proof}
Recall $K=V_{6}(G\cap G_{F_0})$. Write  $\cE_{exp}, \cE_{0.01n},  \cE_{count},A_i$ and $B_i$ for the events  $ \cE_{exp}(G,F_0),$ $\cE_{0.01n}(G,$ $F_0),$ $\cE_{count}(G,$ $F_0),$ $\cA_i(G,F_0)$ and $\cB_i(G,F_0)$ respectively. Let $G_1=([n],E(G)\cap F_0)$. 

$T_A(G,F_0)$ accounts for the running time spent on identifying $K$ (this can be done in $O(n^2)$ time),
running CerHAM modulo lines $7$,$8$ and $22$ (this can be done in $O(n^7)$  time - see the proof of Lemma \ref{lem:upperT}), and executing  CoverAndAdjust$(G_R,S_i,\cP)$ for pairs $(G_R, S_i)$ such that the maximum subset $C$ of $S_i$ with the property that $C\cup N_{G_R}(C)$ is connected in $G_R$ has size less than $\log n/10$. CoverAndAdjust may be executed at most $t+1\leq n+1$ times. At each execution, at most $n$ components that intersect the set $S_i$ in at most $\log n/10$ vertices are considered. Thus each iteration takes $O(n 2^{0.6\log n} \log^6 n )=O(n^2)$ (see last paragraph of Section $2$). Therefore, $T_A(G,F_0)=O(n^7).$

For estimating $\mathbb{E}(T_B(G,F_0))$, first recall from Lemma \ref{lem:upperT} that $T_B(G,F_0)=O(n^72^n).$ Therefore, the first line of \eqref{eq:middleTB} accounts for $\mathbb{E}(T_B\mathbb{I}(\cE_{exp}(G,F_0)\lor \cE_{count}(G,F_0) \lor \cE_{0.01n}(G,F_0)))$.
If the event $\cE_{exp}(G,F_0)\lor \cE_{count}(G,F_0) \lor \cE_{0.01n}(G,F_0)$ does not occur then $|K|\leq 0.01n$ and $|S_t\setminus K|\leq 0.02n$. The second line of \eqref{eq:middleTB} accounts for the running time of FindSparse in the events $n/(10\log_2 n)\leq |K|\leq 0.01n$ and $|S_t\setminus K|\leq 0.02n$ respectively. Finally, the last line accounts for the running time of CoverAndAdjust$(G_R,S_j,\cP)$ for pairs $(G_R, S_i)$ such that the maximum subset $C$ of $S_j$ with the property that $C\cup N_{G_R}(C)$ is connected in $G_R$ has size at least $\log n/10$. Let $C$ be a maximum such set over the pairs $(G_R,S_j)$ and $|C|=i$. Then CoverAndAdjust may be executed at most $n+1$ times, and at each execution at most $n$ components are considered; this yield a running time of $n^2i^62^{3i}$. Now either $\cA_i$ occurs or $\neg \cA_i$. If $\neg \cA_i$ occurs we may partition $C$ into the sets $C_1=C\cap V_6(G_1)$ and $C_2=C\setminus C_1$. Observe that due to the maximality of $C$ and the fact that $K\subseteq S_0$ we have that no vertex in $K\setminus C_1$ is adjacent to $C$, thus $|N_{G_1}(C_1)|\leq 5|C_1|$. In addition $C_2\subseteq S_t\setminus K$. As $\neg \cA_i$ occurs the set $S_t\setminus K$ has size less than $i$. Therefore,  $|N_{G_1}(C_2)|\leq |N_{G_1}(S_t\setminus K)|< 2|S_t\setminus K|\leq 2i$. Hence $|N_{G_1}(C)|\leq |N_{G_1}(C_1)|+|N_{G_1}(C_2)|\leq 5|C_1|+ 2i\leq 5|C|+2i=7i$, i.e the event $\cB_i(G,F_0)$ occurs. 
\end{proof}

\begin{lemma}\label{lem:eq:middle}
 Let $G\sim G(n,p)$,  $G'=H_n$ and $F_0=E(H_n)$. Then, $\mathbb{E}(T_B(G,F_0))=O(n)$ for $\frac{500\log \log n}{n} \leq p\leq \frac{100\log n}{n}$.
\end{lemma}

\begin{proof}
Write  $\cE_{exp}, \cE_{0.01n},  \cE_{count},A_i$ and $B_i$ for the events  $ \cE_{exp}(G,F_0), \cE_{0.01n}(G,$ $F_0),$ $\cE_{count}(G,$ $F_0),$ $\cA_i(G,F_0)$ and $\cB_i(G,F_0)$ respectively. In addition, let $G_1$ and $K$ be as at the beginning of this subsection.  Lemma \ref{lem:pseudo} implies that $(G,F_0)\in \cR$. Thus it suffices to verify that the expression given by \eqref{eq:middleTB} is equal to $O(n)$. This will follow from equations \eqref{eq:big_s}, \eqref{eq:001} and \eqref{eq:002}. 

Recall \eqref{eq:4core} states, $\Pr(\cE_{exp}\lor  \cE_{0.01n})=O(2^{-1.1n}).$ Thereafter Lemma \ref{lem:coundmiddlerandom} states that in the event  $\neg \cE_{exp}\land \neg \cE_{0.01n}$ one has $\sum_{i=1}^{0.01n^2}X_i\leq 2n$. As $X_i=1$ with probability $p$ independently of $X_1,...,X_{i-1}$ the Chernoff bounds give,
\begin{align}\label{eq:001}
     \Pr(\cE_{exp}\lor  \cE_{0.01n} \lor \cE_{count})
     &\leq \Pr(\cE_{exp}\lor  \cE_{0.01n})+
     \Pr(\neg \cE_{exp}\land \neg  \cE_{0.01n} \land \cE_{count})
    \nonumber \\&\leq O(2^{-1.1n})+ \Pr(Bin(0.01n^2,p)\leq 2n )=O(2^{-1.1n}).
\end{align}
Using \eqref{eq:4core}, for $p\geq (500\log\log n)/n$, 
\begin{align}\label{eq:002}
 \sum_{i=\frac{n}{10\log_2 n}}^{0.01n} i^4\binom{n}{i} \Pr\big(|K|=i)   &\leq \sum_{i=\frac{n}{10\log_2 n}}^{0.01n} i^4\bfrac{en}{i}^i\bigg(\bfrac{en}{i} \bfrac{0.1enp}{5}^5 e^{-0.09np}  \bigg)^i \nonumber
\\&\leq \sum_{i=\frac{n}{10\log_2 n}}^{0.01n} i^4 \bigg( \bfrac{n}{i}^2 (0.1np)^5 e^{-0.09pn} \bigg)^i \nonumber
\\&\leq \sum_{i=\frac{n}{10\log_2 n}}^{0.01n} i^4 \bigg( (10\log_2 n)^2 (50\log\log n)^5 e^{-45\log\log n}\bigg)^i=o(1).
\end{align}

Equations \eqref{eq:middleTB}, \eqref{eq:big_s}, \eqref{eq:001}, \eqref{eq:002} and the inequality $n^2i^62^{3i}\leq i^4\binom{n}{i}$ for $(\log n)/10\leq i\leq 0.01n$ imply that $T_B(G, F_0)=O(n)$ for $ \frac{500\log\log n}{n} \leq p\leq \frac{100\log n}{n}$.
\end{proof}

\textbf{Proof of Theorem \ref{thm:main} for $ \frac{500\log\log n}{n}\leq p\leq \frac{100\log n}{n}$:} At the execution of DCerHAM($G$), Lemma \ref{lem:pseudo} implies that $(G,F_0)\in \mathcal{R}$.  Then, the decomposition of the expected running time of DCerHAM($G$) follows from lemmas \ref{lem:middleT} and \ref{lem:eq:middle}.
\qed

\section{The DCerHAM algorithm in the sparse regime}

\subsection{Minimum degree 2}\label{subsec:min2}
De facto, if $G\sim G(n,p)$ does not have minimum degree 2 then $G$ is not Hamiltonian. Such a feature can be detected in $O(n^2)$ time. In the event that $\delta(G(n,p))\geq 2$ we can initialize DCerHam by implementing an algorithm for prepossessing $G$ that runs in time inversely proportional to $\Pr(\delta(G)\geq 2)$.  Thus it becomes crucial to upper bound $\Pr(\delta(G)\geq 2)$. This can be done using the following lemma. Its proof is located at Appendix \ref{subsec:app:min2}. We let $G(n,m)$ be the uniform random graph model. That is for $m,n\in \mathbb{N}$ if $G\sim G(n,m)$ then $G$ is distributed uniformly over all the graphs on $[n]$ with $m$ edges. 
\begin{lemma}\label{lem:degree}
Let $m=O(n\log\log n)$ and $G\sim G(n,m)$. Then,
\begin{equation}\label{eq:mindegree}
    \Pr(\delta(G_{n,m})\geq 2) \leq ne^{- 2m  e^{-\frac{2m}{n}}}.
\end{equation}
\end{lemma}

\subsection{Analysis of DCerHam - the sparse regime}\label{subsec:alg:lower}
Given a graph $G$ and a set of edges $F_0$ such that $(G,F_0)\in \cR$ the DCerHAM algorithm executes the following steps.  First, it checks if the minimum degree of $G$ is $2$. If it is not, then it returns that $G$ is not Hamiltonian. Else it lets $G_1=([n],F_0\cap E(G))$, $c=\max\{2|E(G)|/n,400\}$ and generates $K=V_{6}(G_1)$. Thereafter if $|K| \leq \frac{10n}{c^2}$, then it lets $w=e^{-2c}n$; else it lets $w=c^{-2}n$. Then it runs 
FindSparse($G[V(G)\setminus K],G_1[V(G)\setminus K],E(G_1),\emptyset$) with the twist that at line $4$ it lets $\cQ$ be the set of subsets of $V(G)\setminus S$ of size at most $\max\{w,|S|\}$ and lets $G_R',F',S=FindSparse(G[V(G)\setminus K],G_1[V(G)\setminus K],E(G_1),\emptyset)$. It sets $S_0=K\cup S$ and implements CoverAndAdjust and finds an FPP $\cP_0$ that covers $S_0$ or outputs that $G$ is not Hamiltonian. Finally, 
it continues and executes the CerHAM algorithm with the two alterations given in the previous section. 
 
Let $\cE_K$ be the event that $K>\frac{10n}{c}$.
In place of Lemma \ref{lem:coundmiddlerandom} we have the following one.

\begin{lemma}\label{lem:countlowerroand}
Let $G$ be a graph on $[n]$ and $F_0$ be an edge set on $[n]$ with the property $(G,F_0)\in \cR$.   If the event $\neg \cE_{exp}(G,F_0)\land \neg \cE_{0.01n}(G,F_0) \land \neg \cE_K$ occurs then $\sum_{i=1}^{0.01n^2}X_i\leq 2n$.
\end{lemma}

\begin{proof}
Write  $\cE_{exp}$ and $\cE_{0.01n}$ for the events  $ \cE_{exp}(G,F_0)$ and $\cE_{0.01n}(G,F_0)$  respectively. Let $G_1=([n], E(G)\cap F_0)$ and $K=V_{6}(G_1)$. Let $S_0,S_1,...,S_t$ be as in Lemma \ref{lem:coundmiddlerandom}. 

First, in the event $\neg \cE_{exp} \land \neg \cE_{0.01n} \land \neg \cE_K \land \set{|K| \leq \frac{10n}{c^2}} \land  \set{ |S_0|<1.01|K|}$ as $|K|\leq \frac{10n}{c^2}$ upon termination of FindSparse$(G[V(G)\setminus k],G_1[V(G)\setminus K],E(G_1),w)$ there does not exists a subset $Q$ of $[n]\setminus S_0$ of size at most $\max\{e^{-2c}n,|S_0\setminus K|\} \geq e^{-2c}n$ with the property that $|N_{G_1}(Q)\setminus S_0| <2|Q\setminus S_0|$. Since $W_1$ has this property we have that  $|S_1|-|S_0|= |S_1\setminus S_0|\geq |W_1|\geq e^{-2c}n$. 
\eqref{eq:tlower} implies that  $t\leq \log_2\bfrac{0.02n}{e^{-2c}n}+1\leq 2c\log^{-1}2$. Substituting this last inequality into \eqref{eq:Xilower}, as $c\geq 400$, $|S_0|<1.01|K|\leq \frac{10.1n}{c^2}$ and $|S_t|\leq 0.03n$, we get that,
$$\sum_{i=1}^{0.01n^2}X_i<2\cdot (2c\log^{-1}2) \cdot  \frac{10.1n}{c^2} +4\cdot 0.03n +n<2n.$$

Similarly, in the event $\neg \cE_{exp} \land \neg \cE_{0.01n} \land \neg \cE_K \land  \set{\frac{10n}{c^2}<|K| \leq \frac{10n}{c}} \land  \set{ |S_0|<1.01|K|}$ we have that $|S_0|<1.01|K|\leq \frac{10.1n}{c}$, $|S_1|-|S_0|= |S_1\setminus S_0|\geq |W_1|\geq |K| \geq \frac{10n}{c^2}$ and $t\leq \log_2\bfrac{0.03n}{10c^{-2}n}=\log_2 0.003c^2$.
Therefore,
\begin{align*}
    \sum_{i=1}^{0.01n^2}X_i& \leq 2\cdot \log_2 (0.003c^2) \cdot  \frac{10.1n}{c} +0.12n +n \leq \frac{20.2n \log_2(0.003\cdot 400^2)}{400}+1.12n \leq 2n.
\end{align*}
At the second inequality we used that the function $x^{-1} \log_2 (x)$ is decreasing for $x\geq c\geq  400$.

Finally in the event  $\neg \cE_{exp} \land \neg \cE_{0.01n} \land \neg \cE_K \land  \set{ |S_0|> 1.01|K|}$ we have that upon termination of FindSparse$(G[V(G)\setminus k],G_1[V(G)\setminus K],E(G_1),w)$ there does not exists a subset $Q$ of $[n]\setminus S_0$ of size at most $S_0\setminus K$ with the property that $|N_{G_1}(Q)\setminus S_0| <2|Q\setminus S_0|$. Since $W_1$ has this property we have that  $|S_1|-|S_0|= |S_1\setminus S_0|\geq |S_0\setminus K|\geq 0.01|S_0|$. \eqref{eq:tlower} implies that  $t\leq \log_2\bfrac{0.03n}{0.01|S_0|}$. Substituting this last inequality into \eqref{eq:Xilower} give,
$$\sum_{i=1}^{0.01n^2}X_i\leq 2|S_0|\cdot \log_2\bfrac{3n}{|S_0|} +1.12n \leq 2 \max\{x(\log_2 3 x^{-1}):0\leq x\leq 0.03\}n+1.12n<2n.$$
In all 3 cases we have that $\sum_{i=1}^{0.01n^2}X_i\leq 2n$.
\end{proof}

For a graph $G$ we define the event $\cE_{\delta\geq 2}(G)=\{\delta(G)\geq 2\}.$  Let $T(G,F_0)$ be the running time of DCerHAM$(G,F_0)$, let $T_B(G,F_0)$ be the portion of the running time of DCerHAM spend on executing FindSparse, executing  CoverAndAdjust$(G_R,S_i,\cP)$ for pairs $(G_R, S_i)$ such that the maximum subset $C$ of $S_i$ with the property that $C\cup N_{G_R}(C)$ is connected in $G_R$ has size at least $\log n/10$, and executing line $22$ of CerHAM. Finally, we let $T_A(G,F_0)=T(G,F_0)-T_B(G,F_0)$. In place of Lemma \ref{lem:middleT} we have the following one.

\begin{lemma}\label{lem:lowerT} 
Let $G$ be a graph on $[n]$ and $F_0$ be an edge set on $[n]$ with the property $(G,F_0)\in \cR$. Let $G_1=([n],E(G)\cap F)$  and $c= n|E(G_1)|/(n-1)$. Then, 
with 
$$\cE_{bad}(G,F_0)= \cE_{exp}(G,F_0)\lor \cE_{count}(G,F_0) \lor \cE_{0.01n}(G,F_0) \lor \cE_K\lor \{2|E(G)|/n<400\},$$
we have that $T_A(G,F_0)=O(n^7)$ and 

\begin{align}\label{eq:lowerTB} 
\mathbb{E}(T_{B}(G,F_0))=  
   O\bigg(&  n^72^n\Pr(\cE_{Bad}(G,F_0))+\sum_{i=0}^{0.02n }  i^4 \binom{n}{i}\Pr(\cA_i(G,F_0) ) \nonumber
   \\&+\sum_{i=\frac{\log n}{10}}^{0.02n }n^2i^62^{3i} \Pr(\cA_i(G,F_0))  +\sum_{i=\frac{\log n}{10}}^{0.03n }n^2i^62^{3i} \Pr(\cB_i(G,F_0))  \nonumber
   \\&+ \sum_{\ell=400}^n  n^3\binom{n}{ne^{-2\ell}}\Pr\bigg(\set{ |K| \leq \frac{10n}{\ell^2} }\land \cE_{\delta\geq 2}(G)\land \{c=\ell\} \bigg)\nonumber
 \\&+
 \sum_{\ell=400}^n  n^3\binom{n}{\frac{10n}{\ell^2}}\Pr\bigg(\set{ \frac{10n}{\ell^2} <|K| \leq \frac{10n}{\ell} }\land \{c=\ell\} \bigg)\bigg). 
\end{align}
\end{lemma}
\begin{proof}
Write  $\cE_{\delta \geq 2}, \cE_{exp}, \cE_{0.01n},  \cE_{count},A_i$ and $B_i$ for the events  $\cE_{\delta\geq 2}(G), \cE_{exp}(G,F_0), \cE_{0.01n}(G,$ $F_0),$ $\cE_{count}(G,$ $F_0),$ $\cA_i(G,F_0)$ and $\cB_i(G,F_0)$ respectively. 
It takes $O(n^2)$ time to check whether $\delta(G)\geq 2$. Thereafter the upper bound on $T_A$ follows as in the proof of Lemma \ref{lem:middleT}. 

The justification of the first two lines of \eqref{eq:lowerTB} is almost identical to the justification of \eqref{eq:middleTB} given in the proof of Lemma \ref{lem:middleT}. The last two lines account for the running time of FindSparse in the event $\neg\cE_{Bad} \wedge \{|S_0 \setminus K| \leq w \}$. In this event $c\geq 400$ and $|K|\leq 10n/c$. Thereafter if $|K|\leq 10n/c^2$ then the DCerHAM algorithm sets $w=ne^{-2c}$. Else it sets $w=10n/c^2$. In both cases, at its first execution, FindSparse examines sets of size up to $\max\{w,|S_0\setminus K|\}=w$ and therefore runs in $O(w^3\binom{n}{w})=O(n^3\binom{n}{w})$ time. 
\end{proof}

\begin{lemma}\label{lem:eq:lower}
Let $G\sim G(n,p)$, $G'=E(H_n)$  and $F_0=E(H_n)$. Then, $\mathbb{E}(T_B(G,F_0))=O(n)$ for $\frac{5000}{n} \leq p\leq \frac{500\log\log n}{n}$.
\end{lemma}

\begin{proof}
Lemma \ref{lem:pseudo} implies that $(G,F_0)\in \cR$. Write  $\cE_{\delta\geq 2}, \cE_{exp}, \cE_{0.01n},  \cE_{count},A_i$ and $B_i$ for the events  $\cE_{\delta\geq 2}(G), \cE_{exp}(G,F_0), \cE_{0.01n}(G,$ $F_0),$ $\cE_{count_3}(G,$ $F_0),$ $\cA_i(G,F_0)$ and $\cB_i(G,F_0)$ respectively. In addition let $G_1=([n],E(G)\cap F_0)$.

Lemma \ref{lem:pseudo} implies that $H_n$ spans at least $0.05n^2$ and at most $0.0505n^2+o(n^2)$ edges and therefore $|E(G_1)|\sim Bin(\alpha\binom{n}{2},p)$ for some $a\in [0.1,0.102]$.  \eqref{eq:chernofflower}, \eqref{eq:chernoffupper} give,
\begin{align*}
\Pr(|c-0.1np|> 0.01np)\leq 2e^{- \frac{0.088^2 \cdot 0.1p\binom{n}{2}}{2.088}} \leq 
2e^{-0.8n} \leq 2^{-1.1n}.
\end{align*}
In particular $ \Pr(c<400)=O(2^{-1.1n}).$ Thereafter, using \eqref{eq:big_s}, we have,
\begin{align*}
    \Pr(\cE_K)&\leq \Pr(\cE_K \land \{ |c-0.1np|\leq 0.01np\}) +\Pr(|c-0.1np|> 0.01np)    \leq \Pr\big(\cE_{ \frac{10n}{0.11np}}\big)+2^{-1.1n}
\\&\leq \bigg( \bfrac{n}{\frac{10n}{0.11np}} (0.1np)^5 e^{-0.09np} \bigg)^{\frac{10n}{0.11np}} +2^{-1.1n}
\leq \bigg( 500^6 e^{-450}\bigg)^{\frac{20n}{1.1np}}  +2^{-1.1n} =O(2^{-1.1n}).
\end{align*}
In addition to the above, \eqref{eq:4core} states that $\Pr(\cE_{exp}\lor \cE_{0.01n})=O(2^{-1.1n})$; hence as in the proof of Lemma \ref{lem:eq:middle} one has $\Pr(\cE_{Bad})=O(2^{-1.1n})$. Therefore, as in the proof of Lemma \ref{lem:eq:middle} one has,
\begin{align}\label{eq:101}
 O\bigg(&  n^72^n\Pr(\cE_{Bad}(G,F_0))+\sum_{i=0}^{0.02n }  i^4 \binom{n}{i}\Pr(\cA_i(G,F_0) ) \nonumber
   \\&+\sum_{i=\frac{\log n}{10}}^{0.02n }n^2i^62^{3i} \Pr(\cA_i(G,F_0))  +\sum_{i=\frac{\log n}{10}}^{0.03n }n^2i^62^{3i} \Pr(\cB_i(G,F_0)) \bigg)=O(n).
\end{align}
Thereafter,
\begin{align}\label{eq:102}
   &  \sum_{\ell=400}^n  n^3\binom{n}{ne^{-2\ell}}\Pr\bigg(\set{ |K| \leq \frac{10n}{\ell^2} }\land \cE_{\delta\geq 2}\land \{c=\ell\} \bigg)\nonumber
 \leq \sum_{\ell=400}^n  n^3\binom{n}{ne^{-2\ell}}\Pr\bigg( \cE_{\delta\geq 2}\bigg|\{c=\ell\} \bigg) \nonumber
 \\& \leq \sum_{\ell=400}^n  n^3e^{(2\ell+1)e^{-2\ell}n} ne^{-\ell e^{-(1+o(1))\ell}n}   =O(n).
\end{align}
Finally, as $400<0.09np$ and $\{|K|=i\} \subseteq \cE_{s}=\cE_{s}(G,F_0)$ for $j\leq s$, using \eqref{eq:4core} we get,
\begin{align}\label{eq:103}
 & \sum_{\ell=400}^n n^3\binom{n}{\frac{10n}{\ell^2}}\Pr\bigg(\set{ \frac{10n}{\ell^2} <|K| \leq \frac{10n}{\ell} }\land \{c=\ell\} \bigg) \nonumber
\\& \leq n^3 2^n \Pr(|c-0.1np|\geq 0.01np) 
+ \sum_{\ell=400}^n  n^3\binom{n}{\frac{10n}{\ell^2}}\Pr\bigg(\cE_{\frac{10n}{\ell} }\land \{\ell\in [0.09np,0.11np]\} \bigg) \nonumber
\\& \leq 1 + n^4
\binom{n}{\frac{10n}{(0.09np)^2}}\Pr\big( \cE_{\frac{10n}{(0.11np)^2}} \big) \nonumber
\\&  \leq  1+ \max\set{  n^4\bfrac{(0.09x)^2e}{10}^{\frac{10n}{(0.09x)^2}}  \cdot \bigg( \bfrac{en}{\frac{10n}{(0.11x)^2}}
\bfrac{0.1ex}{5}^5 e^{-0.09x}\bigg)^{\frac{10n}{(0.11x)^2}}      :x\geq 5000}   \nonumber 
\\&  \leq 1+ n^4\max\set{ \big(e^{10}(0.1x)^{10} e^{-0.09x}\big)^{\frac{10n}{(0.11x)^2} }      :x\geq 5000}   \leq 2. 
\end{align}

Equations \eqref{eq:lowerTB}, \eqref{eq:101}, \eqref{eq:102} and \eqref{eq:103} imply that $T_B(G, F_0)=O(n)$ for $ \frac{5000}{n} \leq p\leq \frac{500\log \log n}{n}$.
\end{proof}

\textbf{Proof of Theorem \ref{thm:main} for $ \frac{5000}{n}\leq p\leq \frac{500\log\log n}{n}$:} At the execution of DCerHAM($G$), Lemma \ref{lem:pseudo} implies that $(G,F_0)\in \mathcal{R}$.  Then, the decomposition of the expected running time of DCerHAM($G$) follows from lemmas  \ref{lem:lowerT} and \ref{lem:eq:lower}.
\qed
\section{Conclusion} 
We have introduced and analysed DCerHAM, an algorithm that solves HAM in $O\bfrac{n}{p}$ expected running time over the input distribution $G\sim G(n,p)$ for $p\geq \frac{5000}{n}$. The value $p=\frac{5000}{n}$ has not been optimized; however we believe that new ideas are needed to extend it pass the value $\frac{100}{n}$.
  
\subsection*{Acknowledgement}
The author would like to thank Alan Frieze and the anonymous reviewers for their constructive comments on an earlier version of this manuscript.
 
\bibliographystyle{plain}
\bibliography{bib}

\begin{thebibliography}{10}

\bibitem{ajtai1985}
Mikl{\'o}s Ajtai, J{\'a}nos Koml{\'o}s, and E~Szeraer{\'e}di.
\newblock First occurrence of {Hamilton} cycles in random graphs.
\newblock {\em North-Holland Mathematics Studies}, 115(C):173--178, 1985.

\bibitem{alon2020}
Yahav Alon and Michael Krivelevich.
\newblock Finding a {Hamilton} cycle fast on average using rotations and
  extensions.
\newblock {\em Random Structures \& Algorithms}, 57(1):32--46, 2020.

\bibitem{anastos2021fast}
Michael Anastos.
\newblock Fast algorithms for solving the {Hamilton} cycle problem with high
  probability.
\newblock {\em unpublished, arXiv preprint arXiv:2111.14759}, 2021.

\bibitem{anastos2020}
Michael Anastos and Alan Frieze.
\newblock Hamilton cycles in random graphs with minimum degree at least 3: an
  improved analysis.
\newblock {\em Random Structures \& Algorithms}, 57(4):865--878, 2020.

\bibitem{angluin1979}
Dana Angluin and Leslie~G Valiant.
\newblock Fast probabilistic algorithms for {Hamiltonian} circuits and
  matchings.
\newblock {\em Journal of Computer and system Sciences}, 18(2):155--193, 1979.

\bibitem{bax}
Eric Bax and J.~Franklin.
\newblock A permanent algorithm with $exp[\omega ( n ^{1/3}/2 ln n )]$ expected
  speedup for 0-1 matrices.
\newblock {\em Algorithmica}, 32:157--162, 01 2008.

\bibitem{bellman1962}
Richard Bellman.
\newblock Dynamic programming treatment of the travelling salesman problem.
\newblock {\em Journal of the ACM (JACM)}, 9(1):61--63, 1962.

\bibitem{bjorklund2016below}
Andreas Bj{\"o}rklund.
\newblock Below all subsets for some permutational counting problems.
\newblock In {\em 15th Scandinavian Symposium and Workshops on Algorithm
  Theory}, 2016.

\bibitem{bjorklund2019generalized}
Andreas Bj{\"o}rklund, Petteri Kaski, and Ryan Williams.
\newblock Generalized {Kakeya} sets for polynomial evaluation and faster
  computation of fermionants.
\newblock {\em Algorithmica}, 81(10):4010--4028, 2019.

\bibitem{bogdanov2006}
Andrej Bogdanov and Luca Trevisan.
\newblock Average-case complexity.
\newblock {\em arXiv preprint cs/0606037}, 2006.

\bibitem{bollobas1984}
B{\'e}la Bollob{\'a}s.
\newblock The evolution of sparse graphs, {Graph theory and combinatorics},
  1984.
\newblock {\em MR}, 777163(2):35--57, 1984.

\bibitem{bollobasfennerfrrieze}
B{\'e}la Bollob{\'a}s, Trevor~I. Fenner, and Alan~M. Frieze.
\newblock An algorithm for finding {Hamilton} paths and cycles in random
  graphs.
\newblock {\em Combinatorica}, 7(4):327--341, 1987.

\bibitem{ferber}
Asaf Ferber, Michael Krivelevich, Benny Sudakov, and Pedro Vieira.
\newblock Finding hamilton cycles in random graphs with few queries.
\newblock {\em Random Structures \& Algorithms}, 49(4):635--668, 2016.

\bibitem{frieze2019}
Alan Frieze.
\newblock Hamilton cycles in random graphs: a bibliography.
\newblock {\em unpublished, arXiv preprint arXiv:1901.07139}, 2019.

\bibitem{frieze2016}
Alan Frieze and Micha{l{}} Karo{\'n}ski.
\newblock {\em Introduction to random graphs}.
\newblock Cambridge University Press, 2016.

\bibitem{gurevich1987}
Yuri Gurevich and Saharon Shelah.
\newblock Expected computation time for {Hamiltonian} path problem.
\newblock {\em SIAM Journal on Computing}, 16(3):486--502, 1987.

\bibitem{held1962}
Michael Held and Richard~M Karp.
\newblock A dynamic programming approach to sequencing problems.
\newblock {\em Journal of the Society for Industrial and Applied mathematics},
  10(1):196--210, 1962.

\bibitem{karp1972}
Richard~M Karp.
\newblock Reducibility among combinatorial problems.
\newblock In {\em Complexity of computer computations}, pages 85--103.
  Springer, 1972.

\bibitem{kohn1977generating}
Samuel Kohn, Allan Gottlieb, and Meryle Kohn.
\newblock A generating function approach to the traveling salesman problem.
\newblock In {\em Proceedings of the 1977 annual conference}, pages 294--300,
  1977.

\bibitem{komlos1983}
J{\'a}nos Koml{\'o}s and Endre Szemer{\'e}di.
\newblock Limit distribution for the existence of {Hamiltonian} cycles in a
  random graph.
\newblock {\em Discrete mathematics}, 43(1):55--63, 1983.

\bibitem{korshunov}
Aleksei~Dmitrievich Korshunov.
\newblock Solution of a problem of {Erd{\H{o}}s and Renyi on Hamiltonian}
  cycles in nonoriented graphs.
\newblock In {\em Doklady Akademii Nauk}, volume 228 (3), pages 529--532.
  Russian Academy of Sciences, 1976.

\bibitem{kriv2006}
Michael Krivelevich and Benny Sudakov.
\newblock Pseudo-random graphs.
\newblock In {\em More sets, graphs and numbers}, pages 199--262. Springer,
  2006.

\bibitem{nenadov2020}
Rajko Nenadov, Angelika Steger, and Pascal Su.
\newblock An {$O(N)$} time algorithm for finding {Hamilton} cycles with high
  probability.
\newblock In {\em 12th Innovations in Theoretical Computer Science Conference
  (ITCS 2021)}, volume 185, page~60. Schloss Dagstuhl--Leibniz-Zentrum für
  Informatik, 2021.

\bibitem{posa1976}
Lajos P{\'o}sa.
\newblock Hamiltonian circuits in random graphs.
\newblock {\em Discrete Mathematics}, 14(4):359--364, 1976.

\bibitem{ryser1963combinatorial}
Herbert~John Ryser.
\newblock {\em Combinatorial mathematics}, volume~14.
\newblock American Mathematical Soc., 1963.

\bibitem{Sei}
JJ~Seidel.
\newblock A survey of two-graphs, intern. coll. teorie combinatorie (roma,
  1973), i.
\newblock {\em Accad. Naz. Lincei, Rome}, 1976.

\bibitem{shamir1983}
Eli Shamir.
\newblock How many random edges make a graph {Hamiltonian}?
\newblock {\em Combinatorica}, 3(1):123--131, 1983.

\bibitem{thomason89}
Andrew Thomason.
\newblock A simple linear expected time algorithm for finding a hamilton path.
\newblock {\em Discrete Mathematics}, 75(1-3):373--379, 1989.

\end{thebibliography}

\begin{appendices}

\section{A Family of Pseudorandom  Graphs}\label{app:pseudorandom} 
A graph $G$ is a strongly regular graph with parameters $(n, d, \eta, \mu)$ if (i) $G$ is a $d$-regular graph on $n$ vertices, (ii) for every $x,y\in V(G)$ if $xy \in E(G)$ then $x$ and $y$ have exactly $\eta$ common neighbors and (iii) for every $x,y\in V(G), x\neq y$ if $xy\notin E(G)$ then $x$ and $y$ have exactly $\mu$ common neighbors. For the proofs of the following Lemmas and some further reading on pseudorandom and strongly regular graphs see \cite{kriv2006}.

\begin{lemma}\label{lem:str}
Let $G$ be a connected strongly regular graph with parameters $(n, d, \eta, \mu)$. Then $G$ has only 3 distinct eigenvalues $\lambda_1,\lambda_2,\lambda_3$ which are given by,
\begin{itemize}
    \item $\lambda_1=d$,
    \item $\lambda_2=\frac{1}{2}\bigg(\eta-\mu+\sqrt{(\eta-\mu)^2+4(d-\mu)} \bigg)$ and
    \item $\lambda_3=\frac{1}{2}\bigg(\eta-\mu-\sqrt{(\eta-\mu)^2+4(d-\mu)} \bigg)$.
\end{itemize}
\end{lemma}

For a graph $G$ and disjoint sets $U,W\subset V(G)$ we let $e(U,W)$ be the number of edges with an endpoint in each $U,W$.
\begin{lemma}
Let $G$ be a $d$-regular graph on n vertices. Let $d = \lambda_1\geq \lambda_1 \geq \dots \geq \lambda_n$ be the eigenvalues of $G$. Let $\lambda=\underset{2\leq i\leq n}{\max}|\lambda_i|$. Then for every two disjoint subsets $U, W\subseteq V(G)$,
\begin{equation}\label{eq:app:pseudo}
\bigg|e(U,W)-\frac{d|U||W|}{n} \bigg|\leq \lambda \sqrt{|U||W|\bigg(1-\frac{|U|}{n}\bigg)\bigg(1-\frac{|W|}{n}\bigg)}.    
\end{equation}
\end{lemma}
We now introduce a family of strongly regular graphs $H_{q,k}$ where $1\leq k\leq q$ and $q$ is a prime power which we later use to construct $\{H_n\}_{n\geq 1}$. This family is due to Delsarte and Goethals, and  Turyn (see \cite{kriv2006}, \cite{Sei}). Let $V_q$ be the elements of the two-dimensional vector space over $GF(q)$. Let $L$ be the set of $q+1$ lines that pass through the origin and $L_k$ be a subset of $L$ of size $k$. We define $H_{q,k}$ as the graph on $V_q$ where two elements $x,y \in V_q$ form an edge iff the line that passes through both $x$ and $y$ is parallel to a line in $L_k$. It is easy to check that $H_{q,k}$ is a $k(q-1)$-regular graph on $q^2$ vertices. Every pair of neighbors shares exactly $(k-1)(k-2)+q-1$ neighbors and every pair of non-adjacent vertices shares exactly $k(k-1)$ neighbors. Thus $H_{q,k}$ is strongly regular with parameters $(q^2,k(q-1),(k-1)(k-2)+q-2,k(k-1))$.

Now, for   sufficiently large $n$ we construct the graph $H_n$ as follows: Let $q$ be the smallest prime between $\sqrt{n}$ and $2\sqrt{n}$. Let $k$ be the smallest integer larger than $0.1001q$ and let $H_n'=H_{q,k}$. Now let $H_n''$ be the subgraph of $H_{n}'$ spanned by the first $n$ vertices of $H_{n}'$ which we identify with $[n]$. Join every vertex of degree less than $0.1n$ in $H_n''$ to every vertex in $[0.1n+1]$ and let $H_n$ be the resultant graph.

\begin{lemma}\label{app:lem:pseudo}
For sufficiently large $n$ there exists a graph $H_n$ on $[n]$ that can be constructed in $O(n^7)$ time and satisfies the following.
\begin{itemize}
    \item[(a)] For every pair of disjoint sets $U, W\subseteq [n]$ the number of edges spanned by $U\times W$ is at least $\frac{0.1|U||W|}{n} -n^{7/4}$ and at most $\frac{0.101|U||W|}{n} + n^{7/4}$
    \item[(b)] $H_n$ has minimum degree $0.1n$.
    \item[(c)] At most $n^{2/3}$ vertices in $[n]$ have degree larger than $0.101n$ in $H_n$. 
\end{itemize}
\end{lemma}

\begin{proof}
Determining whether $a$ is a prime takes $O(a^2)$ time by doing a brute force search for divisors of $a$ in $\{2,...,a-1\}$. Thus determining the primes smaller or equal to $4\sqrt{n}$ can be done in $O(n)$, hence $H_n$ can be constructed in $O(n^7)$ time.

Part (b) follows from the construction of $H_n$. Let  $S_n$ and $L_n$ respectively be the set of vertices of degree at most $0.1n$ and at least $0.1005n$ respectively in $H_n''$.

$H_{q,k}$ is strongly regular with parameters $(q^2,k(q-1),(k-1)(k-2)+q-2,k(k-1))$, for some  $k\in[0.1001q,0.1002q]$. Thus, as $|(k-1)(k-2)+q-2-k(k-1)|\leq q$, Lemma \ref{lem:str} implies that $H_{q,k}$ has only 3 distinct eigenvalues $\lambda_1,\lambda_2,\lambda_3$ with $\lambda_1=k(q-1)$ and $|\lambda_2|,|\lambda_3|=O(q)=O(\sqrt{n})$.  Hence \eqref{eq:app:pseudo} implies that for every pair of disjoint subsets  $U,W$ of $V(H_n'')\subseteq V(H_{q,k})$ 
\begin{equation}\label{eq:1}
    \bigg|e_{H_n''}(U,W)-\frac{k(q-1)|U||W|}{q^2} \bigg| =O(n^{3/2}).
\end{equation}
As $k$ is the smallest integer larger than $0.1001q$ for significantly large $n$ we have that $0.1\leq \frac{k(q-1)}{q^2} \leq 0.1002q$.
Taking $U=S_n$ and $W=[n]\setminus U$ in \eqref{eq:1} we get that $|S_n|=O(n^{1/2})$ and $O(n^{3/2})$ edges are added to $H_n''$ to form $H_n$, hence (a) follows.  Similarly, $|L_n|=O(n^{1/2})$. Finally as every vertex of $H_n$ of degree at least $0.101n$ either belongs to $L_n$ or is incident to at least $0.005n$ of the $O(n^{3/2})$ edges added to $H_n''$, we have that there exists $O(n^{1/2})$ such vertices.
\end{proof}

\section{Proof of Lemma \ref{lem:color}}\label{app:sec:6core}
First, let recall Lemma \ref{lem:color}.
\begin{lemma}
Let $\frac{5000}{n}\leq p \leq \frac{500\log\log n}{n}$, $G\sim G(n,p)$, $G'=H_n$ and $F_0=E(G)\cap E(G')$. Then  for $1\leq i\leq 0.02n$, $\frac{\log n}{100} \leq j \leq 0.03n$ and $\frac{n}{\log^2 n}\leq s\leq 0.01n$ the following hold.
\begin{equation}\label{eq:app:4core}
 \Pr(\cE_{s}(G,F_0))\leq  \bigg( \bfrac{en}{s} \bfrac{0.1enp}{5}^5 e^{-0.09np}\bigg)^s, \hspace{10mm} \Pr(\cE_{exp}(G,F_0)\lor \cE_{0.01n})\leq 2^{-1.1n},
\end{equation}
\begin{equation}\label{eq:app:big_s}
i^4 \binom{n}{i}\Pr(\cA_i(G,F_0))=O(1) \hspace{5mm} \text{ and } \hspace{5mm} j^72^{3j}\Pr( \cB_j(G,F_0)) \leq e^{-100j}.
\end{equation}
\end{lemma}

\begin{proof}
Write  $\cE_{exp}, \cE_{0.01n}, \cA_i$ and $\cB_i$ for the events  $ \cE_{exp}(G,F_0), \cE_{0.01n}(G,$ $F_0),$ $\cA_i(G,F_0)$ and $\cB_i(G,F_0)$ respectively. In addition let $G_1=([n],E(G)\cap F_0)$ and $\frac{n}{\log^2 n}\leq s\leq 0.01n$. In the event $\cE_{s}$ there exists a set $S\subset V(G_1)$ of size $s$ such that each vertex in $S$ has at most $5$ neighbors in $[n]\setminus S$. Given a fixed set $S$, as $H_n$ has minimum degree $0.1n$ and $|S|\leq 0.01n$, each vertex in $S$ has at most $5$ neighbors in $[n]\setminus S$ independently with probability at most $$\sum_{i=0}^5 \binom{0.09n}{i}p^i(1-p)^{0.09n-i}\leq \sum_{i=0}^5 (1+o(1))\bfrac{0.09enp}{i}^i e^{-0.09np}\leq  \bfrac{0.1enp}{5}^5 e^{-0.09np}.$$ 
Therefore,  
\begin{align}\label{eq:repeat}
    \Pr(\cE_{s})&\leq \binom{n}{s}\bigg( \bfrac{0.1enp}{5}^5 e^{-0.09np}\bigg)^s\leq \bigg( \bfrac{en}{s} \bfrac{0.1enp}{5}^5 e^{-0.09np}\bigg)^s. 
\end{align}
\eqref{eq:repeat} implies,
$$    \Pr(\cE_{0.01n})\leq \bigg(100e \bfrac{500e}{5}^5e^{-450}\bigg)^{0.01n} \leq 2^{-2n}.$$
In the event $\cE_{exp}\wedge \neg \cE_{0.01}$ there exists disjoint sets $A,B,S\subset V(G)$ such that (i) $V_6(G_1)\subset A$ and $|A|=0.01n$, (ii) $|S|\in [0.02n,0.27n]$ and (iii) $N_{G_1}(S)\setminus A \subseteq B $ and $|B|=2|S|$. Lemma \ref{lem:pseudo} states that there exists at least $0.1s(n-(0.01n+3s))-n^{7/4}$ edges from $S$ to $[n]\setminus (A\cup B\cup S)$. If (i) and (iii) occur, then none of these edges belongs to $G$. Therefore, 
\begin{align*}
\Pr(\cE_{exp} \lor\cE_{0.01n})& \leq \Pr(\cE_{exp}\wedge \neg \cE_{0.01n}) + \Pr(\cE_{0.01n})
\\& \leq \sum_{s=0.02n}^{0.27n}  \binom{n}{0.01n} \binom{n}{s}\binom{n}{2s}(1-p)^{0.1s(0.99n-3s)-n^{7/4}}+2^{-2n}
\\&\leq \sum_{s=0.02n}^{0.27n} (100e)^{0.01n} \bigg(\frac{ene^{-(0.1+o(1))p(0.33n-s)}}{4^{1/3}s} \bigg)^{3s}+2^{-2n}
\\& = \sum_{s=0.02n}^{0.27n} e^{0.01n\big[1+\log 100 +\frac{3s}{0.01n}\big(1 +\log\bfrac{n}{s}-\frac{\log 4}{3}-(0.1+o(1))np\big(0.33-\frac{s}{n}\big)\big]}+2^{-2n}\leq 2^{-1.1n}. 
\end{align*}

In the event $\cA_i \wedge \neg \cE_{0.01n}$ let $W\subset V(G)\setminus V_6(G_1)$ be such that $i\leq |W|\leq 0.02n$ and  $|N_{G_1}(W)\setminus V_6(G)|< 2|W|$. In addition let $Z=N_{G_1}(W)\setminus V_6(G)$ and $A\subset V_6(G_1)$ be maximal with the property that $A\cup W$ spans a forest in $G_1$ where each tree of this forest spans a vertex in $W$. Observe that $W\subset V(G)\setminus V_6(G_1)$ implies that each vertex in $W$ has at least $6$ neighbors in $W\cup Z$ and therefore the set $W\cup Z$ spans at least $(6|W|+|Z|)/2\geq 2|W|+ (|Z|+1+|Z|  )/2 \geq 2|W|+|Z|+1$ edges. In addition, as $H_n$ has minimum degree $0.1n$, for each vertex $w\in W$ none of the at least $0.03n$ edges from $w$ to $V(G_1)\setminus (W\cup Z\cup A)$ appears in $G$. Finally, if we collapse $W$ into a single vertex $w^*$ we have that $A\cup \{w^*\}$ spans a tree in $G_1$. Recall the derivation of \eqref{eq:repeat} and that $\Pr(\cE_{0.01n})\leq 2^{-2n}$ (both are used at the second inequality of the calculations below). 
For $i\leq 0.02n$ we have, 
\begin{align*}
i^4 \binom{n}{i}\Pr(\cA_i) &\leq 
  i^4 \binom{n}{i}\Pr(\cA_i\wedge \neg \cE_{0.01n})+ i^4\binom{n}{i}\Pr(\cE_{0.01n}) \nonumber
  \\&\leq i^4 \binom{n}{i} \sum_{a=0}^{0.01n} \bigg(  \binom{n}{a} (a+1)^{a-1}p^{a}\big((0.02np)^5e^{-0.09np}\big)^a \nonumber
  \\&\hspace{5mm}\times \sum_{w=i}^{0.02n}\sum_{z=0}^{2w} \binom{n}{z+w}\binom{\binom{z+w}{2}}{2w+z+1}p^{2w+z+1}(1-p)^{-0.03wn}\bigg)+ i^4\binom{n}{i}2^{-2n} \nonumber
  \\&\leq i^4 \binom{n}{i} \sum_{a=0}^{0.01n} \bigg( (a+1)(e^2np)^{a}\big((0.02np)^5e^{-0.09np}\big)^a \nonumber
  \\&\hspace{5mm}\times \sum_{w=i}^{0.02n}\sum_{z=0}^{2w} \bfrac{en}{z+w}^{z+w}(1.1(z+w))^{2w+z+1}p^{2w+z+1}e^{-0.03nwp}\bigg)+ 1\nonumber
   \\&\leq i^4 \binom{n}{i} \sum_{a=0}^{0.01n} \bigg( (a+1)(e^2\cdot 5000 \cdot (0.02\cdot 5000)^5e^{-0.09\cdot 5000})^a \nonumber
  \\&\hspace{5mm}\times \sum_{w=i}^{0.02n}\sum_{z=0}^{2w} (enp)^{z+w}(1.1(z+w)p)^we^{-0.03nwp}\bigg)+ 1\nonumber
     \\&\leq i^4 \binom{n}{i} \sum_{w=i}^{0.02n}2w (enp)^{3w}(3.3wp)^{w}e^{-0.03nwp}+ 1\nonumber
     \\&\leq i^4 \bfrac{en}{i}^i \sum_{w=i}^{0.02n}2w \big((5000e)^3 \cdot 3.3wp \cdot e^{-150} \big)^w \leq 3i^5\bfrac{en}{i}^i (ip \cdot e^{-120} )^i  \leq 2.
\end{align*}
Now let $\frac{\log n}{100}\leq j\leq 0.03n$. In the event     $\mathcal{B}_j$ we may identify disjoint sets  $X,W,Z\subset [n]$ such that (i) $|W|=j$, (ii) $Z= N_{G_1}(W)$ and $|Z|\leq 7j$, and 
(iii) $X$ is a minimum size subset of $[n]\setminus (W\cup Z)$ with the property that $X\cup W\cup Z$ is connected in $G$. As $W\cup Z$ spans at most $j=|W|$ components in $G_1$, hence in $G$ we have that (iii) implies, $|X|\leq j-1\leq j$. Finally, Lemma \ref{lem:pseudo} implies that there exist at least $0.09j(n-9j)\geq 0.06nj$ edges from $W$ to $[n]\setminus (X\cup Z\cup W)$ none of which appears in $G$.
Therefore for $(\log n)/100\leq j\leq 0.03n$,
\begin{align*}
    j^72^{3j}\Pr(\cB_j) &\leq j^72^{3j}\sum_{z=0}^{7j}\sum_{x=0}^{j} \binom{n}{j}\binom{n}{z}\binom{n}{x} (j+z+x)^{j+z+x-2}p^{j+x+z-1}  (1-p)^{0.06nj}
    \\ &\leq j^72^{3j}\sum_{z=0}^{7j}\sum_{x=0}^{j} \bfrac{en}{j+x+z}^{j+x+z} (j+z+x)^{j+z+x-2}p^{j+x+z-1}  (1-p)^{0.065nj}
    \\ &\leq j^72^{3j}\sum_{z=0}^{7j}\sum_{x=0}^{j} p^{-1}j^{-2} (enp)^{j+x+z} (1-p)^{0.06nj}
    \leq 7j^72^{3j} (enp)^{9j} e^{-0.06npj} 
    \\&\leq 7j^7\bigg(8(enp)^{9}e^{-0.06np}\bigg)^j 
    \leq 7j^7\bigg(8(5000e)^{9}e^{-300}\bigg)^j \leq 7j^7e^{-200j}\leq e^{-100j}.
\end{align*}
At the last inequality, we used that $j\geq (\log n)/100$.
\end{proof}

\section{Minimum degree 2}\label{subsec:app:min2}
 Let  $G(n,m)$ be the uniform random graph model i.e. if $G\sim G(n,m)$ then $G$ is a graph chosen uniformly at random from all graphs on $[n]$ with $m$ edges. For $m=O(n\log\log n)$ one can generate $G_{n,m}$ via the following process (see \cite{anastos2020}). Let $X_1,X_2,...,X_n$ be i.i.d random variables $Poisson(\lambda)$ and $M$ a multiset that contains $X_i$ copies of $v_i$. Let $X=\sum_{i=1}^n X_i$ and $A=a_1,a_2,...,a_X$ be a random permutation of the elements in $M$. In the event that $X$ is even let $G_A$ be the multi-graph on $\{v_1,v_2,...,v_n\}$ where the edge set is given by $\{a_{2i-1}a_{2i}:1\leq i\leq X/2\}$. In the event that $X=2m$ and $G_A$ is simple we have that $G_A$ has the same distribution as $G_{n,m}$. The probability of $X$ being even, $X=2m$ and $A$ being simple is larger than $\frac{1}{n}$ when $\lambda=\frac{2m}{n}$ and $m=O(n\log\log n)$ (see \cite{frieze2016}, Chapter 11). Therefore,
\begin{align*}
    \Pr(\delta(G_{n,m})\geq 2)&=     \Pr(\delta(G_{A})\geq 2| G_A \text{ is simple}, X=2m)
    \\&= \Pr(\min\{X_1,X_2,...,X_n\}\geq 2|  G_A \text{ is simple }, X=2m) \\&\leq n\Pr(\min\{X_1,X_2,...,X_n\}\geq 2)
    \leq n\bigg(1-\frac{2m}{n} e^{-\frac{2m}{n}}\bigg)^n \leq ne^{- 2m  e^{-\frac{2m}{n}}}.
    \end{align*}
This completes the proof of Lemma \ref{lem:degree}.

\end{appendices}

\end{document}